\documentclass[12pt,oneside]{amsart}
\usepackage{faktor}
\usepackage[hypertexnames=false]{hyperref}
\hypersetup{
	colorlinks=true,
	linktoc=all,
	linkcolor=black,
}

\usepackage{graphicx}
\usepackage{latexsym,amssymb,amscd}
\usepackage{amsmath}
\usepackage[all]{xy}
\usepackage{amsthm}
\usepackage{mathrsfs}
\usepackage{times,enumerate}
\usepackage[usenames,dvipsnames]{color}
\usepackage{blkarray}
\usepackage{amsmath}
\usepackage{faktor}
\usepackage{tikz}

\usetikzlibrary{arrows,decorations.pathmorphing,backgrounds,positioning,fit,matrix}
\usepackage{tikz-cd}
\usepackage{float}
\usepackage{hhline}
\usepackage[english]{babel}

\newcommand\bigw{\scalebox{.85}[1]{$\bigwedge$}}

\usepackage{mathtools}
\usepackage{stackengine}
\setstackEOL{\\}
\usepackage{tikz-cd}
\usetikzlibrary{arrows,decorations.pathmorphing,backgrounds,positioning,fit,petri,calc,shapes.misc,decorations.markings}
\tikzset{degil/.style={
		decoration={markings,
			mark= at position 0.5 with {
				\node[transform shape] (tempnode) {$\backslash$};
			}
		},
		postaction={decorate}
	}
}

      \usepackage{amssymb}

      \theoremstyle{plain}
      \newtheorem{theorem}{Theorem}[section]
      \newtheorem{lemma}[theorem]{Lemma}
      \newtheorem{corollary}[theorem]{Corollary}
      \newtheorem{proposition}[theorem]{Propositition}
      \newtheorem{example}[theorem]{Example}
      
      \theoremstyle{definition}
      \newtheorem{definition}[theorem]{Definition}
      
      \theoremstyle{remark}
      \newtheorem{remark}[theorem]{Remark}

      \makeatletter
      \def\@setcopyright{}
      \def\serieslogo@{}
      \makeatother

      \let\OLDthebibliography\thebibliography
      \renewcommand\thebibliography[1]{
      	\OLDthebibliography{#1}
      	\setlength{\parskip}{8pt}
      	\setlength{\itemsep}{0pt plus 0.3ex}
      }
   
\begin{document}

% First we specify the top matter (author, title, etc).
%
% Note: All of the top matter items are optional and can be omitted.
% But you will probably want to specify at least the author and title
% and perhaps an abstract.

   % author information

   % first author 

   \author{Spyridon Afentoulidis-Almpanis}
 
   \address{Department of Mathematics, Bar-Ilan University, Ramat-Gan, 5290002, Israel}

   \email{spyridon.almpanis@biu.ac.il}
   
   % title

   \title[Dirac cohomology
   for the BGG category $\mathcal{O}$]{Dirac cohomology
   	for the BGG category $\mathcal{O}$}

   % Note that the short title for running heads goes in square
   % brackets.  This is optional.  The long title goes in curly
   % braces.  In the long title, line breaks are indicated by \\.

   % abstract (optional)
   \begin{abstract}
     We study Dirac cohomology $H_D^{\mathfrak{g},\mathfrak{h}}(M)$ for modules belonging to category $\mathcal{O}$ of a finite dimensional complex semisimple Lie algebra. We start by studying the generalized infinitesimal character decomposition of $M\otimes S$, with $S$ being a spin module of $\mathfrak{h}^\perp$. As a consequence, "Vogan's conjecture" holds, and we prove a nonvanishing result for $H_D^{\mathfrak{g},\mathfrak{h}}(M)$ while we show that in the case of a Hermitian symmetric pair $(\mathfrak{g},\mathfrak{k})$ and an irreducible unitary module $M\in\mathcal{O}$, Dirac cohomology coincides with the nilpotent Lie algebra cohomology with coefficients in $M$. In the last part, we show that the higher Dirac cohomology and index introduced by Pand\v zi\'c and Somberg satisfy nice homological properties for $M\in\mathcal{O}$.
   \end{abstract}

   % AMS subject classifications (used in AMS journals)

   % AMS keywords (used in AMS journals)

   % acknowledge support, etc
   \thanks{The author was supported by the Israel Science Foundation (grant No. 1040/22) and the
   	Grantov\'a agentura \v Cesk\'e republiky grant EXPRO
   	19-28628X}
   %\thanks{We would like to thank our colleagues for their helpful
%     criticism.}

   % ication
  % \dedicatory{Dedicated to XXXXXXXXXXX}

   % today's date, or fill in whatever date you prefer
   %\date{\today}

% This ends the top matter information.
% We can now tell LaTeX to display the top matter.

   \maketitle

% Having displayed the top matter, we now proceed to the body of the
% article.

% The body of the article is divided into sections.
% Each section begins with a \section command.

\tableofcontents

   \section{Introduction}\label{introduction}
%   Let $\mathfrak{g}$ be a complex semisimple Lie algebra, $\mathfrak{t}$ a Cartan subalgebra of $\mathfrak{g}$, $U(\mathfrak{g})$ the universal enveloping algebra of $\mathfrak{g}$ and $X$ a simple $(\mathfrak{g},K)$-module. Then, $X$ admits an infitesimal character, i.e. the center $Z(\mathfrak{g})$ of $U(\mathfrak{g})$ acts by a morphism $\chi_\lambda:Z(\mathfrak{g})\rightarrow \mathbb{C}$ on $X$.
 %  Using Harish-Chandra's morphism, one can show that the various $\chi$ are parametrized by elements $\lambda\in\mathfrak{t}^*$. We say that $\chi$, or by abuse of terminology the corresponding element $\lambda$, is the infinitesimal character of $X$.
   
  % One crucial 
  % In \cite{pandzic}, Huang and Pand\v zi\'c proved a 
  %In this paper, we study Dirac cohomology for modules belonging to Berstein-Gelfand-Gelfand category $\mathcal{O}$, i.e. for finitely generated, locally $\mathfrak{n}$-finite weight modules.
  
  Algebraic Dirac operators were introduced by Vogan in a series of lectures at MIT in the late 90's. More presicely, if $\mathfrak{g}$ is a complex semisimple Lie algebra and $\mathfrak{k}$ a maximal compact subalgebra of $\mathfrak{g}$, to each $\mathfrak{g}$-module $(X,\pi)$, he associated an operator $D_{\mathfrak{g},\mathfrak{k}}(X):X\otimes S\rightarrow X\otimes S$ given by
   \begin{equation}\label{standform}
   D_{{\mathfrak g},{\mathfrak k}}(X)=\sum_j \pi(Z_j)\otimes\gamma(Z_j)
   \end{equation}
   %Σκληρός δρόμος
   where $\{Z_j\}$ is an orthonormal basis of ${\mathfrak p}:=\mathfrak{k}^\perp$, the orthogonal complement of $\mathfrak{k}$ with respect to the Killing form $\langle\cdot,\cdot\rangle$ of $\mathfrak{g}$, $S$ is a spin module for the Clifford algebra $\mathbf{C}(\mathfrak{p})$ of $\mathfrak{p}$ and $\gamma$ stands for the Clifford multiplication of $\mathbf{C}(\mathfrak{p})$ on $S$. The operator $D_{\mathfrak{g},\mathfrak{k}}(X)$ is well defined in the sense that it does not depend on the choice of the basis $\{Z_j\}$. The tensor product $X\otimes S$ has a $\mathfrak{k}$-action via the diagonal embedding 
     \begin{equation*}
   (\cdot)_\Delta:{\mathfrak k}\longrightarrow U({\mathfrak g})\otimes\mathbf{C}({\mathfrak p})
   \end{equation*}
   of ${\mathfrak k}$ into $U({\mathfrak g})\otimes\mathbf{C}({\mathfrak p})$, given by the embedding of ${\mathfrak k}$ in $U({\mathfrak g})$ and the embedding \eqref{haction} of $\mathfrak{k}$ in $\mathbf{C}({\mathfrak p})$, while, with respect to this action, $D_{\mathfrak{g},\mathfrak{k}}(X)$ is $\mathfrak{k}$-equivariant, so that $\ker D_{\mathfrak{g},\mathfrak{k}}(X)$ and $\mathrm{im}\hspace{1mm}D_{\mathfrak{g},\mathfrak{k}}(X)$ are $\mathfrak{k}$-modules. 
   
   Vogan defined the Dirac cohomology of a $\mathfrak{g}$-module $X$ to be the $\mathfrak{k}$-module 
   \begin{equation}\label{diraccohom}
   H_D^{{\mathfrak g},{\mathfrak k}}(X)=\frac{\ker D_{{\mathfrak g},{\mathfrak k}}(X)}{\mathrm{im}\hspace{0.5mm}D_{\mathfrak{g},\mathfrak{k}}(X)\cap\ker D_{{\mathfrak g},{\mathfrak k}}(X)}.
   \end{equation}
   Assuming that $K$ is a compact group such that $X$ is a $(\mathfrak{g},K)$-module, $H_D^{\mathfrak{g},\mathfrak{k}}(X)$ is a module for the spin $2$-cover $\widetilde{K}$ of $K$. Vogan conjectured that if $H_D^{{\mathfrak g},{\mathfrak k}}(X)$ 
   contains a $\widetilde{K}$-submodule with highest weight $\beta$, then $V$ has infinitesimal character $\beta+\rho_{\mathfrak k}$.
   In other words, the infinitesimal character of a module can be recovered from its Dirac cohomology. Vogan's conjecture was first proved by Huang and Pand\v zi\'c in \cite{huangpandzic}. Crucial for the proof of Vogan's conjecture (as well for many other results and applications) is the fact that $2D_{\mathfrak{g},\mathfrak{k}}(X)^2$ differs from the Casimir by a scalar. Namely,
   \begin{equation}\label{diracsquare}
   2D_{{\mathfrak g},{\mathfrak k}}(X)^2=\Omega_{\mathfrak g}\otimes 1-(\Omega_{\mathfrak k})_\Delta+\lVert\rho\rVert^2-\lVert\rho_{\mathfrak k}\rVert^2
   \end{equation}
   where $\Omega_{\mathfrak g}$ (respectively $\Omega_{\mathfrak k}$) denotes the Casimir element in the enveloping algebra $U({\mathfrak g})$ (respectively $U({\mathfrak k})$) of $\mathfrak{g}$ (respectively $\mathfrak{k}$), $\rho$ (respectively $\rho_{\mathfrak k}$) is half the sum of the positive roots for some fixed positive system $\Delta^+$ (respectively $\Delta_\mathfrak{k}^+$) of the root system $\Delta$ (respectively $\Delta_\mathfrak{k}$) of $\mathfrak{g}$ (respectively $\mathfrak{k}$) and $\lVert\cdot\rVert$ is the norm induced by the Killing form of ${\mathfrak g}$. 
   
   In the case of equal rank $\mathfrak{g}$ and $\mathfrak{k}$, another interesting invariant of $X$ can be defined using $D_{\mathfrak{g},\mathfrak{k}}(X)$. More precisely, in this case the spin module $S$ splits as $\mathfrak{k}$-module into $S=S^+\oplus S^-$ and $D_{\mathfrak{g},\mathfrak{k}}(X)$ interchanges $X\otimes S^\pm$ so that there are operators
   \begin{equation*}
D_{\mathfrak{g},\mathfrak{k}}(X)^\pm: X\otimes S^\pm\rightarrow X\otimes S^\mp.
   \end{equation*}
   Then, we can define ''half''-Dirac cohomology spaces 
   \begin{equation*}
   H_D^{\mathfrak{g},\mathfrak{k}}(X)^\pm:=\frac{\ker D_{{\mathfrak g},{\mathfrak k}}(X)^\pm}{\mathrm{im}\hspace{0.5mm}D_{\mathfrak{g},\mathfrak{k}}(X)^\mp\cap\ker D_{{\mathfrak g},{\mathfrak k}}(X)^\pm}
   \end{equation*}
  and the virtual $\mathfrak{k}$-module 
   \begin{equation*}
   I_D(X):=H_D^{\mathfrak{g},\mathfrak{k}}(X)^+-H_D^{\mathfrak{g},\mathfrak{k}}(X)^-
   \end{equation*}
   known as the Dirac index of $X$. In \cite{diracindex}, Mehdi, Pand\v zi\'c and Vogan proved a translation principle for $I_D(X)$ and provided interesting applications. 
   % {\color{blue}More precisely, for the sake of simplicity suppose that $\mathfrak{g}$ and $\mathfrak{k}$ are equal rank and fix a Cartan subalgebra $\mathfrak{t}$ of $\mathfrak{g}$ in $\mathfrak{k}$. Let $\Delta:=\Delta(\mathfrak{g},\mathfrak{t})$ (respectively $\Delta_\mathfrak{k}$) be the corresponding root system of $\mathfrak{g}$ (respectively $\mathfrak{k}$) and fix $\Delta^+\subset\Delta$ (respectively $\Delta_\mathfrak{k}^+:=\Delta^+\cap\Delta_\mathfrak{k}$
  % to be a positive system of $\Delta^+$ (respectively $\Delta_\mathfrak{k}$).
 %  contains a $\widetilde{K}_0$-module with highest weight $\beta$ with respect to some fixed positive system $\Delta^+_\mathfrak{k}\subset \Delta^+$ of $\mathfrak{k}$, then $X$ has infinitesimal character $\beta+\rho_{\mathfrak k}$, where $\rho_\mathfrak{k}$ stands for half the sum of the positive roots in $\Delta_{\mathfrak k}^+\subset\Delta^+$. }
   
   In \cite{goette} and \cite{Kostant-1999}, Goette and independently Kostant generalized the above construction of Dirac operators to more general pairs $(\mathfrak{g},\mathfrak{h})$ where $\mathfrak{h}$ is a subalgebra of $\mathfrak{g}$ such that 
   \begin{itemize}
   	\item[(a)]$\mathfrak{h}$ is reductive
   	\item[(b)] the restriction of the Killing form of $\mathfrak{g}$ to $\mathfrak{h}$ is nondegenerate. 
   	\end{itemize}
   	Then, $\mathfrak{g}$ decomposes into an orthogonal direct sum
   \begin{equation*}
   \mathfrak{g}=\mathfrak{h}\oplus\mathfrak{q},
   \end{equation*}
   where $\mathfrak{q}:=\mathfrak{h}^\perp$ with respect to the Killing form of $\mathfrak{g}$, and we consider again a spin module $S$ for the Clifford algebra $\mathbf{C}(\mathfrak{q})$ of $\mathfrak{q}$. For a $\mathfrak{g}$-module $X$, they defined the corresponding Dirac operator $D_{\mathfrak{g},\mathfrak{h}}(X):X\otimes S \rightarrow X\otimes S$, known as cubic Dirac operator, to be given by
 \begin{equation}\label{cubicDirac}
 D_{{\mathfrak g},{\mathfrak h}}(X)=\sum_j \pi(Z_j)\otimes\gamma(Z_j)-1\otimes\gamma(c),
 \end{equation}
 where $\{Z_j\}$ is an orthonormal basis of ${\mathfrak q}$ and $\gamma(c)$ is given by
 \begin{equation}\label{cubictermm}
 \gamma(c)=\frac{1}{6}\sum_{i,j,k} B(Z_i,[Z_j,Z_k])\gamma(Z_i)\gamma(Z_j)\gamma(Z_k).
 \end{equation}

   We note that in the case where $\mathfrak{h}$ is a symmetric Lie algebra, $\gamma(c)=0$.
   The operator $D_{\mathfrak{g},\mathfrak{h}}(X)$ turns out to have the same "good" properties with $D_{\mathfrak{g},\mathfrak{k}}(X)$. Namely, $D_{\mathfrak{g},\mathfrak{h}}(X)$  
   does not depend on the choice of the basis $\{Z_j\}$ while it is $\mathfrak{h}$-equivariant with respect to the action of $\mathfrak{h}$ on $X\otimes S$ given by a diagonal embedding of $\mathfrak{h}$ in $U(\mathfrak{g})\otimes\mathbf{C}(\mathfrak{q})$. The modified formula, with the extra cubic term $1\otimes \gamma(c)$, is used instead of formula \eqref{standform} so that we have the convenient formula for $D_{\mathfrak{g},\mathfrak{h}}(X)^2$:
   \begin{equation}\label{Dsquare}
   2D_{\mathfrak{g},\mathfrak{h}}(X)^2=\Omega_\mathfrak{g}\otimes1-(\Omega_\mathfrak{h})_\Delta+\lVert\rho\rVert^2-\lVert\rho_\mathfrak{h}\rVert^2,
   \end{equation}
   where $\rho_\mathfrak{h}$ is half the sum of the positive roots of the positive system $\Delta_\mathfrak{h}^+:=\Delta_\mathfrak{h}\cap\Delta^+$ for $\Delta_\mathfrak{h}:=\Delta(\mathfrak{h},\mathfrak{t})$.
   For a more detailed discussion about the case where the cubic term $\gamma(c)$ is omitted and some applications in Differential Geometry, see \cite{thesis} and \cite{afentoulidis}. Moreover, as we did in the case of the symmetric pair $(\mathfrak{g},\mathfrak{k})$, we can still define the Dirac cohomology $H_D^{\mathfrak{g},\mathfrak{h}}(X)$ of a $\mathfrak{g}$-module $X$ to be the $\mathfrak{h}$-module
   \begin{equation*}
    H_D^{{\mathfrak g},{\mathfrak h}}(X)=\frac{\ker D_{{\mathfrak g},{\mathfrak h}}(X)}{\mathrm{im}\hspace{0.5mm}D_{\mathfrak{g},\mathfrak{h}}(X)\cap\ker D_{{\mathfrak g},{\mathfrak h}}(X)}.
   \end{equation*}

   In the case where $X$ is a finite-dimensional $\mathfrak{g}$-module, $H_D^{{\mathfrak g},{\mathfrak h}}(X)$ coincides with $\ker D_{\mathfrak{g},\mathfrak{h}}(X)$. If, moreover, $\mathfrak{h}$ contains a Cartan subalgebra $\mathfrak{t}$ of $\mathfrak{g}$, Kostant gave a complete decomposition of $\ker D_{\mathfrak{g},\mathfrak{h}}(X)$ into simple $\mathfrak{h}$-modules:
   \begin{equation}\label{kernel}
   \ker(D_{\mathfrak{g},\mathfrak{h}}(X))=\bigoplus_{w\in W^1} F_{w(\lambda+\rho)-\rho_{\mathfrak{h}}},
   \end{equation}
   where $\lambda\in\mathfrak{t}^*$ is the highest weight of $X$, $F_\mu$ stands for the simple highest weight $\mathfrak{h}$-module with extremal weight $\mu\in\mathfrak{t}^*$, and
   \begin{equation*}
   W^1:=\{w\in W\mid \Delta_\mathfrak{h}^+\subseteq w\Delta^+\},
   \end{equation*}
   with $W$ being the Weyl group of $\mathfrak{g}$.
   In fact, a similar formula was proved in the unequal rank situation, in \cite{kang-pandzic} for the pair $({\mathfrak g},{\mathfrak k})$ and in \cite{Mehdi-2014} for pairs $({\mathfrak g},{\mathfrak h})$. In particular, in the unequal rank case, the decomposition of $H_D^{{\mathfrak g},{\mathfrak h}}(X)$ into irreducibles is no longer multiplicity free.
   
   Apart from finite dimensional modules, Dirac cohomology has been studied for various families of modules, including highest weight modules in \cite{HX}, $A_{\mathfrak q}(\lambda)$ modules in \cite{kang-pandzic}, generalized Enright-Varadarajan modules in \cite{MP10}, unipotent representations in \cite{BP2} and \cite{BP1}. In all the above cases, some finiteness assumption has been made on the restriction of the corresponding $\mathfrak{g}$-module to $\mathfrak{k}$ (or more generally to $\mathfrak{h}$). For instance, in the case of a $(\mathfrak{g},K)$-module $X$, by definition, the restriction of $X$ to $\mathfrak{k}$, admits an algebraic direct sum decomposition into simple finite dimensional $\mathfrak{k}$-modules. 
   
   The aim of this paper is to study Dirac cohomology for modules $M$ belonging to Berstein-Gelfand-Gelfand category $\mathcal{O}$. Category $\mathcal{O}$, introduced by J. Bernstein, I. Gelfand and S. Gelfand in \cite{BGG}, consists of the finitely generated, locally $\mathfrak{n}$-finite weight modules of $\mathfrak{g}$. Examples of modules belonging to $\mathcal{O}$ include finite dimensional modules, Verma modules, (twisted) dual Verma modules, etc. This category has nice homological properties and it seems to be the "correct" module category to study questions raised by Verma \cite{verma} concerning composition series and embeddings of Verma modules, and Jantzen \cite{jantzen} concerning his so-called translation functors. 
   
     In \cite{HX}, Huang and Xiao computed Dirac cohomology in terms of Kazhdan-Lusztig polynomials in the case of modules belonging to the parabolic category $\mathcal{O}^\mathfrak{p}$; here, $\mathfrak{p}$ is a standard parabolic subalgebra of $\mathfrak{g}$ with Levi decomposition $\mathfrak{p}=\mathfrak{l}\oplus\mathfrak{u}$, and $\mathfrak{h}:=\mathfrak{l}$. The parabolic category $\mathcal{O}^\mathfrak{p}$ is the full subcategory of $\mathcal{O}$ consisting of the $\mathfrak{g}$-modules $M\in\mathcal{O}$ such that the restriction $M_{\mid\mathfrak{l}}$ of $M$ to the Levi factor $\mathfrak{l}$ decomposes into an algebraic direct sum of finite dimensional $\mathfrak{l}$-submodules. Hence, there is again a finiteness condition satisfied for $M_{\mid\mathfrak{l}}$.
     
   The main difference in our case is that $M$ belongs to the larger category $\mathcal{O}$, so that
   $M_{\mid\mathfrak{h}}$ need not satisfy any finiteness condition and $\mathfrak{h}$ need not be a Levi factor of some parabolic. For instance, the case of a Verma module $M(\lambda)$ could be illustrative. Whenever $\mathfrak{h}$ strictly contains a Cartan subalgebra $\mathfrak{t}$ of $\mathfrak{g}$, any $\mathfrak{h}$-submodule of $M(\lambda)_{\mid\mathfrak{h}}$ turns out to be infinite dimensional. Indeed, for such an $\mathfrak{h}$, there will be some $A\subset \Delta^+$ such that $\mathfrak{h}$ decomposes into
   \begin{equation*}
   \mathfrak{h}=\mathfrak{t}\oplus\mathfrak{n}_\mathfrak{h}^+\oplus\mathfrak{n}_\mathfrak{h}^-,
   \end{equation*}   
   where
   \begin{equation*}
   \mathfrak{n}_\mathfrak{h}^\pm=\bigoplus\limits_{\alpha\in A}\mathfrak{g}_{\pm\alpha},
   \end{equation*}
   and
   \begin{equation*}
   \mathfrak{g}=\mathfrak{h}\oplus\mathfrak{n}_\mathfrak{q}^+\oplus\mathfrak{n}_\mathfrak{q}^-
   \end{equation*}
   with
   \begin{equation*}
   \mathfrak{n}_\mathfrak{q}^\pm:=\bigoplus\limits_{\alpha\in\Delta\setminus A}\mathfrak{g}_{\pm\alpha}.
   \end{equation*}
   Then, as $\mathfrak{n}_\mathfrak{h}^-$-modules, 
   \begin{align*}
   M(\lambda)_{\mid\mathfrak{n}_\mathfrak{h}^-}=U(\mathfrak{g})\otimes_{U(\mathfrak{t}\oplus\mathfrak{n}_\mathfrak{h}^+\oplus\mathfrak{n}_\mathfrak{q}^+)} \mathbb{C}_\lambda\cong U(\mathfrak{n}_\mathfrak{h}^-\oplus\mathfrak{n}_\mathfrak{q}^-)\cong U(\mathfrak{n}_\mathfrak{h}^-)\otimes_\mathbb{C}U(\mathfrak{n}_\mathfrak{q}^-).
   \end{align*}
    In other words, $M(\lambda)_{\mid \mathfrak{n}_\mathfrak{h}^-}$ is a free $U(n_\mathfrak{h}^-)$-module so that every $U(\mathfrak{n}_\mathfrak{h}^-)$-submodule is infinite dimensional. This should be the case for every $U(\mathfrak{h})$-submodule of $M(\lambda)_{\mid\mathfrak{h}}$. 
    Another problem that may arise is the fact that the restriction $M_{\mid\mathfrak{h}}$ of $M\in\mathcal{O}$ is not finitely generated in general. For example, one can consider the case where $M$ is a Verma module $M(\lambda)$ and $\mathfrak{h}$ coincides with the Cartan subalgebra $\mathfrak{t}$ of $\mathfrak{g}$. For a thorough discussion about restrictions of Verma modules, e.g., \cite{kobayshiVerma,kobayashisoucek,somberg}. The following example illustrates the above difficulties.
    
    \begin{example}
    Let $\mathfrak{g}:=\mathfrak{sl}(3,\mathbb{C})$ and choose the Cartan subalgebra $\mathfrak{t}$ of $\mathfrak{g}$ to be
    \begin{equation*}
    	\mathfrak{t}:=\{H:=\mathrm{diag}(t_1,t_2,t_3)\mid t_1+t_2+t_3=0,t_i\in\mathbb{C}\}
    \end{equation*} 
With the above notation:
    \begin{equation*}
    	\Delta=\{\pm(\varepsilon_1-\varepsilon_2),\pm(\varepsilon_1-\varepsilon_3),\pm(\varepsilon_2-\varepsilon_3)\},
    \end{equation*}
where $\varepsilon_i(H):=t_i$, $i=1,2,3$.
The root vector corresponding to the root $\varepsilon_i-\varepsilon_j$ is the element $e_{i,j}$ of the standard basis of $\mathfrak{gl}(3,\mathbb{C})$ where all the entries are $0$ apart from the entry $(i,j)$ which is $1$. We note that the elements $h_{12}:=e_{11}-e_{22}$ and $h_{23}:=e_{22}-e_{33}$\hspace{1mm} form a basis of $\mathfrak{t}$. We choose the standard positive root system
\begin{equation*}
    	\Delta^+:=\{\varepsilon_1-\varepsilon_2,\varepsilon_1-\varepsilon_3,\varepsilon_2-\varepsilon_3\}
    \end{equation*}
so that 
\begin{equation*}
	\mathfrak{n}=\mathfrak{g}_{\varepsilon_1-\varepsilon_2}\oplus\mathfrak{g}_{\varepsilon_1-\varepsilon_3}\oplus\mathfrak{g}_{\varepsilon_2-\varepsilon_3}
\end{equation*}
and 
\begin{equation*}
	\mathfrak{n}^-=\mathfrak{g}_{-\varepsilon_1+\varepsilon_2}\oplus\mathfrak{g}_{-\varepsilon_1+\varepsilon_3}\oplus\mathfrak{g}_{-\varepsilon_2+\varepsilon_3}.
\end{equation*}
Let
\begin{equation*}
	\mathfrak{h}:=\mathfrak{t}\oplus\mathfrak{g}_{\varepsilon_1-\varepsilon_2}\oplus\mathfrak{g}_{-\varepsilon_1+\varepsilon_2}
\end{equation*}
so that 
	\begin{equation*}\mathfrak{n}_\mathfrak{h}=\mathfrak{g}_{\varepsilon_1-\varepsilon_2}\hspace{2mm}\text{ and }\hspace{2mm}\mathfrak{n}_\mathfrak{h}^-=\mathfrak{g}_{-\varepsilon_1+\varepsilon_2}.
	 	\end{equation*}
Consider the Verma module $M:=M(-\rho)$ with highest weight $-\rho=-\varepsilon_1+\varepsilon_3\in\mathfrak{t}^*$, i.e. the $U(\mathfrak{g})$-module 
\begin{equation*}
	M=U(\mathfrak{g})\otimes_{U(\mathfrak{t}\oplus\mathfrak{n})}\mathbb{C}_{-\rho},
\end{equation*}
where $\mathbb{C}_{-\rho}$ is the $1$-dimensional left module of $U(\mathfrak{t}\oplus\mathfrak{n})$ where $\mathfrak{t}$ acts by $-\rho$ and $\mathfrak{n}$ trivially. The algebra $U(\mathfrak{g})$ has a PBW vector space basis \cite{Knapp1} consisting of vectors of the form
\begin{equation}\label{pbwform}
	e_{21}^{n_1}e_{31}^{n_2}e_{32}^{n_3}h_{12}^{l_1}h_{23}^{l_2}e_{12}^{m_1}e_{13}^{m_2}e_{23}^{m_3},
\end{equation}
where $n_i,l_i,m_i\geq 0$, so that $M$ is spanned by all elements of the form
\begin{equation}\label{formnice}
	e_{21}^{n_1}e_{31}^{n_2}e_{32}^{n_3}\otimes 1.
	\end{equation}
%On the other hand, $I_{-\rho}$ consists of all elements of the form \eqref{pbwform} with $l_1+l_2+l_3+m_1+m_2+m_3>0$. 
%Therefore,
%$M$, as vector space, is isomorphic with $U(\mathfrak{n}^-)$, i.e. with the vector space spanned by all vectors of the form
%\begin{equation}\label{sform}
%	e_{21}^{n_1}e_{31}^{n_2}e_{32}^{n_3}.
%\end{equation}
%This isomorphism turns out to be a $U(\mathfrak{n}^-)$-module isomorphism. 
Let us, now, show that any nontrivial $\mathfrak{h}$-submodule of the restriction $M_{\mid\mathfrak{h}}$ of $M$ is infinite-dimensional. Namely, let $N$ be an $\mathfrak{h}$-submodule of $M_{\mid\mathfrak{h}}$ and let $v$ be a nonzero vector of $N$. Since $\mathfrak{t}\subset\mathfrak{h}$, $N$ is a weight module so that, without loss of generality, we can assume that $v$ is of the form \eqref{formnice}.
 Since $\mathfrak{n_\mathfrak{h}}^-\subset \mathfrak{h}$, the subspace $L:=U(\mathfrak{n}_\mathfrak{h}^-)\{v\}$ is contained in $N$ while $L$ is infinite-dimensional. Indeed, $L$ contains all linear independent vectors of the form
\begin{equation*}
		e_{21}^{n_1+k}e_{31}^{n_2}e_{32}^{n_3}\otimes 1
	\end{equation*}
with $k\geq0$, so that $L$, and hence $N$, is infinite-dimensional. 

Let us, now, show that $M_{\mid\mathfrak{h}}$, as $U(\mathfrak{h})$-module, is not finitely generated. For the sake of contradiction, assume that it is generated by finitely many vectors $v_1,\ldots,v_k$. Without loss of generality, we may assume that these vectors are weight vectors. Then, from the representation theory of $\mathfrak{sl}(2,\mathbb{C})$, one deduces that every weight subspace $M_\mu$, $\mu\in\mathfrak{t}^*$, satisfies $\dim M_\mu\leq k$. Indeed, each vector $v_i$ of the above generating set contributes by an, at most, $1$-dimensional vector space to $M_\mu$, so that $\dim M_\mu\leq k$. This is a contradiction, since $\dim M_\mu=k+1>k$ for $\mu=-k(\varepsilon_1-\varepsilon_3)$. More precisely, all linearly independent vectors of the form 
\begin{equation*}
	e_{21}^{n}e_{31}^{k-n}e_{32}^{n}\otimes 1\in M, \hspace{3mm} k=0,\ldots,n
	\end{equation*}
belong to $M_{\mu}$. Hence the $U(\mathfrak{h})$-module $M_{\mid\mathfrak{h}}$ is not finitely generated.

Let us, now, study the Dirac operator 
\begin{equation*}
D_{\mathfrak{g},\mathfrak{h}}(M) :M\otimes S\rightarrow M\otimes S, 
	\end{equation*}
where $S=\mathbb{C}\{1,e_{31},e_{32},e_{31}\wedge e_{32}\}$ is the spin module of $\mathfrak{h}$ (see Section \ref{Spinreprese}). The maximal weight of $S$ is $\frac{1}{2}\varepsilon_1+\frac{1}{2}\varepsilon_2-\varepsilon_3$ and corresponds to $1\in S$. Up to some coefficient, the operator $D_{\mathfrak{g},\mathfrak{h}}(M)$ is given by 
\begin{equation*}
	D_{\mathfrak{g},\mathfrak{h}}(M)=e_{13}\otimes\gamma(e_{31})+e_{23}\otimes\gamma(e_{32})+e_{31}\otimes\gamma(e_{13})+e_{32}\otimes\gamma(e_{23}).
\end{equation*}
Let $v^+:=1\otimes 1\in M$. Then, $v^+\otimes 1 \in M\otimes S$ is the maximal weight vector of the $U(\mathfrak{h})$-module $M\otimes S$ and, since $e_{13}v^+=e_{23}v^+=0$ and $\gamma(e_{13})1=\gamma(e_{23})1=0$,
\begin{equation*}
	D_{\mathfrak{g},\mathfrak{h}}(M)(v^+\otimes 1)=0.
\end{equation*}
 Therefore, $N:=U(\mathfrak{h})\{v^+\otimes 1\}\subset \ker D_{\mathfrak{g},\mathfrak{h}}(M)$. We note that $N$ is a highest weight module of highest weight \begin{equation*}\mu:=-\frac{1}{2}\varepsilon_1+\frac{1}{2}\varepsilon_2
\end{equation*} and 
\begin{equation*}
	2\frac{\langle \mu,\varepsilon_1-\varepsilon_2\rangle}{\lVert \varepsilon_1-\varepsilon_2\rVert^2}=-1\notin \mathbb{Z}_{>0}
\end{equation*}
so that, according to \cite[Theorem 4.8]{humO}, $N$ is the simple Verma module $M_{\mathfrak{h}}(\mu)$ of $\mathfrak{h}$. Then, the natural projection of $N$ in $H_D^{\mathfrak{g},\mathfrak{h}}(M)$ is injective and we can show that $H_D^{\mathfrak{g},\mathfrak{h}}(M)$ is isomorphic with $N$. Namely, according to Theorem \ref{voganO}, $H_D^{\mathfrak{g},\mathfrak{h}}(M)$ has finitely many composition factors while, according to Corollary \ref{vogy}, if $N'$ is another submodule of $H_D^{\mathfrak{g},\mathfrak{h}}(M)$, it should be again of maximal weight $\mu$. On the other hand, the weight subspace $(M\otimes S)_\mu$ is $1$-dimensional so that $H_D^{\mathfrak{g},\mathfrak{h}}(M)\cong M_\mathfrak{h}(\mu)$.
%and set $V:=U(\mathfrak{n}_\mathfrak{h})\{v_1,\ldots,v_k\}$ with $d:=\dim_\mathbb{C}V$. Since $U(\mathfrak{h})\cong U(\mathfrak{n}_\mathfrak{h}^-)\otimes U(\mathfrak{t})\otimes U(\mathfrak{n}_\mathfrak{h})$, we obtain $M=U(\mathfrak{n}_\mathfrak{h}^-)V$. We claim that 

% $U(\mathfrak{n}_\mathfrak{h}^-)$ is abelian and can be viewed as the polynomial algebra of $e_{12}$, and $M$ as $U(\mathfrak{n}_\mathfrak{h}^-)$-module is isomorphic with $U(\mathfrak{n}_\mathfrak{h}^-)\otimes \mathbb{C}[e_{31},e_{32}] $, where the $U(\mathfrak{n}_\mathfrak{h}^-)$-action is only on the first factor. In other words, 
    \end{example}
   
 %  \begin{example}
  % 	Let $\mathfrak{g}:=\mathfrak{sl}(3,\mathbb{C})$, $\mathfrak{t}$ a Cartan subalgebra of $\mathfrak{g}$ and $\Delta:=\{\varepsilon_i-\varepsilon\mid 1\leq i\neq j\leq 3\}$ the corresponding root system. Assume that $\mathfrak{h}:=\mathfrak{t}\oplus\mathfrak{g}_{\varepsilon_1-\varepsilon_2}\otimes \mathfrak{g}_{-\varepsilon_1+\varepsilon_2}$, where $\mathfrak{g}_{\alpha}$, $\alpha\in\Delta$ stands for the corresponding root subspace of $\alpha$.
  % \end{example}

   In this paper, we extend the methods used in the study of Dirac cohomology of $(\mathfrak{g},K)$-modules in order to study Dirac cohomology $H_D^{\mathfrak{g},\mathfrak{h}}(M)$ in the case where $M$ belongs to $\mathcal{O}$. More precisely, after collecting the basic notions about Clifford algebras and spinors in Section \ref{Spinreprese}, in Section \ref{Dirac cohomology and BGG} we show that, even if the decomposition of $M\otimes S$ into $\mathfrak{h}$-submodules can be tricky, $M\otimes S$ has a generalized eigenspace decomposition for the operator $D_{\mathfrak{g},\mathfrak{h}}(M)^2$ (Corollary \ref{decom}). As a consequence, Dirac cohomology $H_D^{\mathfrak{g},\mathfrak{h}}(M)$ is finitely generated and Vogan's conjecture holds (Theorem \ref{voganO}). Moreover, we show that Dirac cohomology does not vanish whenever $M\neq 0$ (Theorem \ref{novanish}). 
   As a direct application, in Section \ref{Application}, we show that the Dirac cohomology of a simple Verma module is again a simple Verma module for $\mathfrak{h}$ (Theorem \ref{simple}).
   
   Formula \eqref{kernel} resembles Kostant's formula for nilpotent Lie algebra cohomology with coefficients in a finite dimensional simple module \cite{kostakis}. 
   Indeed, in \cite{renardpandzic}, Huang, Pand\v zi\'c  and Renard related nilpotent Lie algebra cohomology and Dirac cohomology in the case of Hermitian symmetric pairs $(\mathfrak{g},\mathfrak{k})$ and irreducible unitary $(\mathfrak{g},K)$-modules. In Section \ref{DiracandLie}, we show that a similar result holds when $X$ is replaced by an irreducible unitary module in $\mathcal{O}$ (Theorem \ref{cohom}).
  
 In \cite{somberg}, for equal rank $\mathfrak{g}$ and $\mathfrak{h}$, Pand\v zi\'c and Somberg studied Dirac cohomology for $(\mathfrak{g},K)$-modules having generalized infinitesimal character and satisfying some finiteness condition for $\mathfrak{h}$. In contrast to the case of a $\mathfrak{g}$-module $X$ having infinitesimal character where $D_{\mathfrak{g},\mathfrak{h}}(X)^2$, by \eqref{Dsquare}, acts on each $\mathfrak{h}$-isotypic component of $X\otimes S$ by a scalar, when $X$ has a generalized infinitesimal character, this is not true in general. As a consequence, $\ker D_{\mathfrak{g},\mathfrak{h}}(X)^2$ need not exhaust the generalized kernel of $D_{\mathfrak{g},\mathfrak{h}}(X)$ and $H_D^{\mathfrak{g},\mathfrak{h}}(X)$ may not have nice Homological Algebra properties. In order to overcome this problem, they introduced a notion of higher Dirac cohomology and proved that this new $\mathfrak{h}$-module has nice homological properties. In the last section, we adapt their methods in order to prove similar results for $H_D^{\mathfrak{g},\mathfrak{h}}(M)$ when $M$ is a module in $\mathcal{O}$.
 
 $\quad$\\
 \textbf{Acknowledgements:} The author would like to thank Professors S. Mehdi, P. Pand\v zi\'c and V. Sou\v{c}ek for their suggestions which have substantially improved the presentation of this paper. The author would like to thank the anonymous referees for their valuable comments.

% by the scalar $\lVert\Lambda\rVert-\lVert\mu+\rho_\mathfrak{h}\rVert$ on the $K$-type of highest weight $\mu\in\mathfrak{t}^*$

  \section{Spin module}\label{Spinreprese} 
  In this section, we recall the basic notions of Clifford algebras and spinors. More details can be found in \cite[Chapter 6]{goodman} and \cite[Chapter 2]{pandzic}.
  Let $\mathfrak{g}$ be a complex semisimple Lie algebra, $\langle\cdot,\cdot\rangle$ its Killing form, $\theta$ a Cartan involution of $\mathfrak{g}$, and $\mathfrak{h}$ a Lie subalgebra of $\mathfrak{g}$ such that:
   \begin{equation}\label{conditionh}
   \begin{cases}
   \mathfrak{h}\text{ is reductive}\\
   \mathfrak{h}\text{ is $\theta$-stable}\\
   \langle\cdot,\cdot\rangle_{\mid\mathfrak{h}\times\mathfrak{h}} \text{ is nondegenerate.}
   \end{cases}
   \end{equation}
   Then, $\mathfrak{g}$ decomposes into a direct sum
   \begin{equation*}
   \mathfrak{g}=\mathfrak{h}\oplus\mathfrak{q},
   \end{equation*}   
   where $\mathfrak{q}:=\mathfrak{h}^\perp$ with respect to $\langle\cdot,\cdot\rangle$. Consider the Clifford algebra $\mathbf{C}(\mathfrak{q})$ of $\mathfrak{q}$, i.e. the quotient of the tensor algebra $T(\mathfrak{q})$ of $\mathfrak{q}$ with the ideal generated by all elements 
   \begin{equation*}
   X\otimes Y+Y\otimes X-\langle X,Y\rangle, \quad X,Y\in\mathfrak{q}.
   \end{equation*} 
   
   For $a$ and $b$ in $\mathfrak{q}$, let $R_{a,b}$ be the element of 
   \begin{equation*}\mathfrak{so}(\mathfrak{q}):=\{T\in \mathfrak{gl}(\mathfrak{q})\mid \langle Tu,v\rangle+\langle u,Tv\rangle=0,\forall u,v\in \mathfrak{q} \}.
   \end{equation*}
   defined by 
   \begin{equation*}
   R_{a,b}(v):=\langle b,v\rangle a-\langle a,v\rangle b, \quad v\in \mathfrak{q}.
   \end{equation*}
   Then $\{R_{a,b}\mid a,b\in\mathfrak{q}\}$ spans $\mathfrak{so}(\mathfrak{q})$ \cite[Lemma 6.2.1]{goodman} and there is a Lie algebra monomorphism $\varphi:\mathfrak{so}(\mathfrak{q})\rightarrow \mathbf{C}(\mathfrak{q})$ given by
   \begin{equation*}
   \varphi:R_{a,b}\mapsto \frac{1}{2}[\gamma(a),\gamma(b)].
   \end{equation*} 
   Let $\mathfrak{q}^-$ be a maximal isotropic subspace of $\mathfrak{q}$ and set $S:=\bigw\mathfrak{q}^-$. Then, $\mathbf{C}(\mathfrak{q})$ acts on $S$ by Clifford multiplication while $S$ equipped with the $\mathfrak{h}$-action defined by
   \begin{equation}\label{haction}
   \mathfrak{h}\overset{\mathrm{ad}}{\longrightarrow} \mathfrak{so}(\mathfrak{q})\overset{\varphi}{\longrightarrow} \mathbf{C}(\mathfrak{q})\longrightarrow \mathrm{End}(S)
   \end{equation}
   is a Lie algebra module called spin module of $\mathfrak{h}$. We note that the spin module $S$ does not depend on the choice of $\mathfrak{q}^-$ \cite[Theorem 6.1.3]{goodman}. 
   %Moreover, $S$ admits a ..... Hermitian form $\langle\cdot,\cdot\rangle_S$ such that 
  % \begin{equation*}
  % \langle\gamma(X)u,v\rangle=\langle u,\gamma(X)v\rangle,\quad\forall X\in\mathfrak{q},\forall u,v\in S.
 %  \end{equation*}
   
   Suppose that there is a common Cartan subalgebra $\mathfrak{t}$ of $\mathfrak{g}$ and $\mathfrak{h}$. Then, the root system $\Delta:=\Delta(\mathfrak{g},\mathfrak{t})$ splits into $\Delta=\Delta_\mathfrak{h}\sqcup\Delta_\mathfrak{q}$ so that 
   \begin{equation*}
   \mathfrak{h}=\mathfrak{t}\oplus\bigoplus_{\alpha\in\Delta_\mathfrak{h}}\mathfrak{g}_\alpha \text{ and }
   \mathfrak{q}=\bigoplus_{\beta\in\Delta_\mathfrak{q}}\mathfrak{g}_\beta.
   \end{equation*}
   Choose a positive system $\Delta^+\subset\Delta$ for $\Delta$ and set $\Delta_\mathfrak{h}^+:=\Delta^+\cap\Delta_\mathfrak{h}$ to be a positive system for the root system $\Delta_\mathfrak{h}$ and $\Delta_\mathfrak{q}^+:=\Delta^+\cap\Delta_\mathfrak{q}$ while $\rho$ and $\rho_\mathfrak{h}$ stand for the half-sum of positive roots in $\Delta$ and $\Delta_\mathfrak{h}$ respectively. Then, 
   \begin{equation}\label{qpositive}
   \mathfrak{n}^{\pm}_\mathfrak{q}:=\bigoplus_{\beta\in\Delta_\mathfrak{q}^+}\mathfrak{g}_{\pm\beta}.
   \end{equation}
   are maximal dual isotropic subspace of $\mathfrak{q}$ and $S$ can be chosen to be 
   \begin{equation*}
   S:=\bigw \mathfrak{n}_\mathfrak{q}^-.
   \end{equation*}
   
   Then, as $\mathfrak{h}$-module, $S$ is a weight module, i.e. it can be decomposed into a direct sum of its weight subspaces. More precisely, let $\beta_1,\ldots,\beta_l$
   be the positive roots lying in $\mathfrak{q}$ and $e_{\pm 1},\ldots,e_{\pm l}$ the corresponding root vectors of $\pm\beta_1,\ldots,\pm\beta_l$ respectively, such that 
   
   \begin{equation*}
   \langle e_i,e_j\rangle=
   \begin{cases}
   1&\text{if }i+j=0\\
   0& \text{otherwise.}
   \end{cases}
   \end{equation*}
   If $I=(i_1,\ldots,i_s)$ with $1\leq s\leq l $, set 
   \begin{equation*}
   u_I:=e_{-i_1}\wedge\ldots\wedge e_{-i_s}\in S.
   \end{equation*}
   Then $u_I$ is a weight vector of $S$, the weight being
   \begin{equation}\label{Sweight}\rho-\rho_\mathfrak{h}-\sum\limits_{i\in I}\beta_i=\frac{1}{2}\big\{\sum\limits_{i\notin I}\beta_i-\sum\limits_{i\in I}\beta_i\big\},
   \end{equation}
   where, by abuse of notation, we write $i\in I$ if $i\in\{i_1,\ldots,i_s\}$. One easily checks that as $\mathfrak{t}$-modules 
   \begin{equation*}
   S\cong \bigw \mathfrak{n}_\mathfrak{q}^-\otimes \mathbb{C}_{\rho-\rho_\mathfrak{h}}.
   \end{equation*}
    Moreover, from the $\mathbb{Z}_2$-splitting of the vector space $\bigwedge \mathfrak{n}_\mathfrak{q}^-$, the vector space $S$ admits a $\mathbb{Z}_2$-splitting 
   \begin{equation}\label{splitting}
   S=S^+\oplus S^-.
   \end{equation}
   Under the assumption of equal rank $\mathfrak{g}$ and $\mathfrak{h}$, this splitting is preserved by the action of $\mathfrak{h}$.

   \section{Dirac cohomology and BGG category $\mathcal{O}$}\label{Dirac cohomology and BGG}
   Let $\mathfrak{g}$ be a complex semisimple Lie algebra and $\mathfrak{h}$ a Lie subalgebra of $\mathfrak{g}$ satisfying \eqref{conditionh}. Moreover, assume that $\mathfrak{h}$ contains a Cartan subalgebra $\mathfrak{t}$ of $\mathfrak{g}$.
 In this section, we show that for a module $M\in\mathcal{O}$, $M\otimes S$ admits a direct sum decomposition in terms of generalized $\mathfrak{h}$-infinitesimal character subspaces. As a corollary, $M\otimes S$ has a generalized eigenspace decomposition for the operator $D_{\mathfrak{g},\mathfrak{h}}(M)^2$ while Vogan's conjecture holds for $H_D^{\mathfrak{g},\mathfrak{h}}(M)$. At the end of the section, we show that whenever $M\neq0$, $H_D^{\mathfrak{g},\mathfrak{h}}(M)$ is nonzero.
 
 Let $M$ be a module belonging to $\mathcal{O}$. Then, $M$ can be decomposed into a direct sum
 \begin{equation}\label{chardecomp}
 M=\bigoplus\limits_{\lambda}M^{\chi_\lambda},
 \end{equation}
 where $\chi_\lambda:Z(\mathfrak{g})\rightarrow \mathbb{C}$ is the infinitesimal character determined by $\lambda\in\mathfrak{t}^*$ using the Harish-Chandra map \cite[Theorem 5.62]{Knapp1}, and 
 \begin{equation}\label{charr}
 M^{\chi_\lambda}:=\{v\in M\mid \forall z\in Z(\mathfrak{g}), \exists \hspace{1mm}n:=n(z,v)>0 \text{ s.t. }(z-\chi_\lambda(z))^nv=0\}
 \end{equation}
 is the generalized infinitesimal character subspace of $\lambda$. In other words, category $\mathcal{O}$ decomposes into a direct sum
 \begin{equation*}
 \mathcal{O}=\bigoplus\limits_{\lambda}\mathcal{O}_{\chi_\lambda}
 \end{equation*}
 with $\mathcal{O}_{\chi_\lambda}$ being the full subcategory of $\mathcal{O}$ consisting of modules $M\in\mathcal{O}$ such that $M^{\chi_\lambda}=M$ \cite{humO}.
 For the study of $D_{\mathfrak{g},\mathfrak{h}}(M)$, it would be convenient to have a similar decomposition into generalized infinitesimal character subspaces for the $\mathfrak{h}$-module $M\otimes S$. Let $\xi_\nu:Z(\mathfrak{h})\rightarrow \mathbb{C}$ be the character of $Z(\mathfrak{h})$ determined by $\nu\in\mathfrak{t}^*$ \cite[Theorem 5.62]{Knapp1}.
 
 %Let $\mathrm{maxspec}$ denote the set of all characters of $Z(\mathfrak{h})$, i.e. the set of all unital algebra morphisms $\xi:Z(\mathfrak{h})\rightarrow \mathbb{C}$. We denote by $\xi_\nu$ the element of $\mathrm{maxspec}$ which corresponds to $\nu\in\mathfrak{t}^*$ \cite[Theorem 5.62]{Knapp1}.
 
 %Nevertheless, although $M_{\mid\mathfrak{h}}$ is a locally $n_\mathfrak{h}$-finite weight module for 

 %Nevertheless, the restriction $M_\mathfrak{h}$ of $M$ to $\mathfrak{h}$ need not
 
 %Although, as we have already seen in Section \ref{introduction}, $M_{\mid\mathfrak{h}}$ and so $M\otimes S$ need not belong to the corresponding category $\mathcal{O}(\mathfrak{h})$ of $\mathfrak{h}$, the following lemma provides a similar decomposition in terms of generalized infinitesimal character subspaces for $M\otimes S$.
   
   \begin{lemma}\label{le1}
   	Let $M\in\mathcal{O}$. Then $M\otimes S$ can be decomposed into a direct sum 
   	\begin{equation*}
   	M\otimes S=\bigoplus\limits_{\nu\in\mathfrak{t}^*}\big(M\otimes S\big)^{\xi_\nu}
   	\end{equation*}
   	where $\big(M\otimes S\big)^{\xi_\nu}$ stands for the generalized infinitesimal character subspace corresponding to $\xi_\nu$. Moreover, each $\big(M\otimes S\big)^{\xi_\nu}$ is finitely generated.
   \end{lemma}
   
   \begin{proof}[Proof of Lemma \ref{le1}]
   	Since the actions of $Z(\mathfrak{h})$ and $\mathfrak{t}$ on $M\otimes S$ commute, it suffices to prove the statement for any weight subspace $\big(M\otimes S\big)_\mu$, $\mu\in\mathfrak{t}^*$, of $M\otimes S$. Since $M$ belongs to the category $\mathcal{O}$, $\big(M\otimes S\big)_\mu$ is finite-dimensional and the result follows from a standard Linear Algebra argument. 
   	
   	For the second statement, if $\big(M\otimes S\big)^{\xi_\nu}=0$, there is nothing to prove. If not, let $v_1$ be a maximal weight vector of $\big(M\otimes S\big)^{\xi_\nu}$ and set \begin{equation*}X_1:=U(\mathfrak{h})v_1\subseteq\big(M\otimes S\big)^{\xi_\nu}.
   	\end{equation*}
   	If $\big(M\otimes S\big)^{\xi_\nu}=X_1$, there is nothing more to prove. If not, let $\bar{v}_2$ be a maximal weight vector of $\big(M\otimes S\big)^{\xi_\nu}/X_1$ and set $\bar{X}_2:=U(\mathfrak{h})\bar{v}_2$. Let $X_2$ be the preimage of $\bar{X}_2$ in $\big(M\otimes S\big)^{\xi_\nu}$.
   	Inductively, if  $\big(M\otimes S\big)^{\xi_\nu}/X_{i-1}$ is nontrivial, we define $\bar{v}_i$ to be a maximal weight of $\big(M\otimes S\big)^{\xi_\nu}/X_{i-1}$, $\bar{X}_i$ the cyclic submodule generated by $\bar{v}_i$ and $X_i$ the preimage of $\bar{X}_i$ in $\big(M\otimes S\big)^{\xi_\nu}$. Therefore, we obtain a filtration 
   	\begin{equation}\label{filtration1}
   	0=X_0\subsetneq X_1\subsetneq \ldots\subseteq\big(M\otimes S\big)^{\xi_\nu}.
   	\end{equation}
   	It suffices to show that \eqref{filtration1} is finite. Suppose that $\big(M\otimes S\big)^{\xi_\nu}/X_{i-1}$ is not trivial and let $\bar{v}_i=v_i+X_{i-1}$ be a nonzero maximal weight vector of it, say of weight $\nu_i$. Then, $Z(\mathfrak{h})$ acts on it by the character $\xi_{\nu_i+\rho_\mathfrak{h}}$ and so $\xi_{\nu_i+\rho_\mathfrak{h}}$ must coincide with $\xi_\nu$ so that $\nu$ and $\nu_i+\rho_\mathfrak{h}$ belong to the same $W$-orbit. Since there are only finitely many weights $\nu_i$ of $M\otimes S$ satisfying this condition and for such a weight the corresponding weight subspace $\big(M\otimes S\big)_{\nu_i}$ is finite dimensional, we deduce that filtration \eqref{filtration1} is of finite length and so $\big(M\otimes S\big)^{\xi_\nu}$ must be finitely generated.
   \end{proof}

Thanks to \eqref{Dsquare}, we have the following lemma.

	\begin{lemma}\label{le2}
	Let $M\in\mathcal{O}_{\chi_\lambda}$. Then
	$\big(M\otimes S\big)^{\xi_\nu}$ is a subspace of the generalized eigenspace of $D_{\mathfrak{g},\mathfrak{h}}(M)^2$  with eigenvalue 
	\begin{equation*}
	c_\nu:=\frac{1}{2}\big(\chi_\lambda(\Omega_\mathfrak{g})+\lVert\rho\rVert^2-\xi_\nu(\Omega_{\mathfrak{h}})-\lVert\rho_\mathfrak{h}\rVert^2\big)=\frac{1}{2}\big(\lVert\lambda+\rho\rVert^2-\lVert\mu+\rho_\mathfrak{h}\rVert^2\big).
	\end{equation*}
	
\end{lemma}

\begin{proof}[Proof of Lemma \ref{le2}]
	Let $x\in \big(M\otimes S \big)^{\xi_\nu}$. Then, there are $n_1,n_2\geq1$ such that 
	\begin{equation*}
	\big(\Omega_\mathfrak{g}\otimes1-\chi_\lambda(\Omega_\mathfrak{g})\big)^{n_1}x=0
	\end{equation*}
	and 
	\begin{equation*}
	\big((\Omega_\mathfrak{h})_\Delta-\xi_\nu(\Omega_{\mathfrak{h}})\big)^{n_2}x=0.
	\end{equation*}
	Set \begin{equation*}c_\nu:=\frac{1}{2}\big(\chi_\lambda(\Omega_\mathfrak{g})+\lVert\rho\rVert^2-\xi_\nu(\Omega_{\mathfrak{h}})-\lVert\rho_\mathfrak{h}\rVert^2\big).
	\end{equation*}
	We will show that 
	\begin{equation*}
	\big(D_{\mathfrak{g},\mathfrak{h}}(M)^2-c_\nu\big)^{n_1+n_2-1}x=0.
	\end{equation*}
Recall that according to \eqref{Dsquare},
 \begin{equation*}
D_{\mathfrak{g},\mathfrak{h}}(X)^2=\frac{1}{2}\big(\Omega_\mathfrak{g}\otimes1-(\Omega_{\mathfrak{h}})_\Delta+\lVert\rho\rVert^2-\lVert\rho_\mathfrak{h}\rVert^2\big)
\end{equation*}
while $\Omega_\mathfrak{g}\otimes 1$ and $(\Omega_\mathfrak{h})_\Delta$ commute. Therefore
\begin{align*}
	\big(2D_{\mathfrak{g},\mathfrak{h}}(M)^2-2c_\nu\big)^{n_1+n_2-1}=&\Big(\big(\Omega_\mathfrak{g}\otimes1-\chi_\lambda(\Omega_\mathfrak{g})\big)-\big(\Omega_{\mathfrak{h}_\Delta}-\xi_\nu((\Omega_\mathfrak{h})_\Delta)\big)\Big)^{n_1+n_2-1}\\=&\sum\limits_{k=0}^{n_1+n_2-1}(-1)^k\binom{n_1+n_2-1}{k}\big(\Omega_\mathfrak{g}\otimes1-\chi_\lambda(\Omega_\mathfrak{g})\big)^k\\
	&\big((\Omega_\mathfrak{h})_\Delta-\xi_\nu(\Omega_\mathfrak{h})\big)^{n_1+n_2-1-k}
\end{align*}
and one can easily check that the exponent of $\big(\Omega_\mathfrak{g}\otimes1-\chi_\lambda(\Omega_\mathfrak{g})\big)$ is  greater that $n_1$ or the exponent of $\big((\Omega_\mathfrak{h})_\Delta-\xi_\nu(\Omega_\mathfrak{h})\big)$ is greater that $n_2$. Hence, 
$\big(\Omega_\mathfrak{g}\otimes1-\chi_\lambda(\Omega_\mathfrak{g})\big)^kx=0$
or
$
\big((\Omega_\mathfrak{h})_\Delta-\xi_\nu(\Omega_\mathfrak{h})\big)^{n_1+n_2-1-k}x=0$,
and 
\begin{equation*}\big(D_{\mathfrak{g},\mathfrak{h}}(M)^2-c_\nu\big)^{n_1+n_2-1}x=0.\end{equation*}

\end{proof}

Let $\big(M\otimes S\big)(c)$ stand for the generalized eigenspace of $D_{\mathfrak{g},\mathfrak{h}}(M)^2$ with eigenvalue $c\in\mathbb{R}$.
Combining, now, Lemmas \ref{le1} and \ref{le2}, we obtain the following corollary.

\begin{corollary}\label{decom}
	Let $M\in\mathcal{O}$. Then
	$M\otimes S$ can be decomposed into a direct sum 
	\begin{equation*}
	M\otimes S=\bigoplus\limits_{c\in\mathbb{R}}\big(M\otimes S\big)(c)
	\end{equation*}
	of generalized eigenspaces $\big(M\otimes S\big)(c)$ of $D_{\mathfrak{g},\mathfrak{h}}(M)^2$. Moreover, every generalized eigenspace $\big(M\otimes S\big)(c)$ is finitely generated.
\end{corollary}
   
   \begin{proof} Without loss of generality, we assume that $M\in\mathcal{O}_{\chi_\lambda}$. According to Lemma \ref{le1}, $\big(M\otimes S\big)$ decomposes into a direct sum of generalized infinitesimal character subspaces
   	\begin{equation*}
   	M\otimes S=\bigoplus\limits_{\nu\in\mathfrak{t}^*}\big(M\otimes S\big)^{\xi_\nu}
   	\end{equation*}
   	while from Lemma \ref{le2}, $\big(M\otimes S\big)^{\xi_\nu}$ is a subspace of the generalized eigenspace $\big(M\otimes S\big)(c_\nu)$ of $D_{\mathfrak{g},\mathfrak{h}}(M)^2$  with eigenvalue 
   	\begin{equation*}
   		c_\nu:=\frac{1}{2}\big(\chi_\lambda(\Omega_\mathfrak{g})+\lVert\rho\rVert^2-\xi_\nu(\Omega_{\mathfrak{h}})-\lVert\rho_\mathfrak{h}\rVert^2\big)=\frac{1}{2}\big(\lVert\lambda+\rho\rVert^2-\lVert\mu+\rho_\mathfrak{h}\rVert^2\big).
   	\end{equation*}
    Hence, $M\otimes S$ decomposes into a direct sum of generalized eigenspaces $\big(M\otimes S\big)(c)$ for $D_{\mathfrak{g},\mathfrak{h}}(M)^2$ while 
   	\begin{equation*}
   	\big(M\otimes S\big)(c)=\bigoplus\limits_{\substack{\nu\text{ s.t.}\\c=c_\nu}}\big(M\otimes S\big)^{\xi_\nu}.
   	\end{equation*}
   	
   	For the second statement, since there are finitely many weights $\nu$ of $M\otimes S$ satisfying the condition $c_\nu=c$, the above direct sum must be finite. According to Lemma \ref{le1}, every $\big(M\otimes S\big)^{\xi_\nu}$ is finitely generated and so this is the case for $\big(M\otimes S\big)(c)$.   	\end{proof}
   
  % \begin{definition}
  % 	Let $M\in\mathcal{O}$. The $\mathfrak{h}$-module
   	%\begin{equation*}
   	%H_D^{\mathfrak{g},\mathfrak{h}}(M):=\bigslant{\ker D_{\mathfrak{g},\mathfrak{h}}(M)}{\big(\ker D_{\mathfrak{g},\mathfrak{h}}(M)\cap\mathrm{im}\hspace{0.5mm}D_{\mathfrak{g},\mathfrak{h}}(M)\big)}
   	%\end{equation*}
   	%is called the Dirac cohomology of $M$.
   %\end{definition}
  
  %As we saw in Section \ref{introduction}, Dirac cohomology
 % \begin{equation*}
  %H_D^{{\mathfrak g},{\mathfrak h}}(M)=\frac{\ker D_{{\mathfrak g},{\mathfrak h}}(M)}{\mathrm{im}\hspace{0.5mm}D_{\mathfrak{g},\mathfrak{h}}(M)\cap\ker D_{{\mathfrak g},{\mathfrak h}}(M)}.
    %\end{equation*}
    
    Let $\mathcal{O}(\mathfrak{h})$ be the BGG category of the reductive Lie algebra $\mathfrak{h}$. Corollary \ref{decom} results in the following theorem.
  
  \begin{theorem}\label{voganO}
  	Let $M\in\mathcal{O}$. Then $H_D^{\mathfrak{g},\mathfrak{h}}(M)$ is finitely generated so that it belongs to the BGG category $\mathcal{O}(\mathfrak{h})$ for $\mathfrak{h}$. 
  	\end{theorem}

\begin{proof}
	For the first statement, we note that $\ker D_{\mathfrak{g},\mathfrak{h}}(M)$ is a submodule of the generalized eigenspace $\big(M\otimes S\big)(0)$ which according to Corollary \ref{decom} is finitely generated. Since the universal enveloping algebra $U(\mathfrak{h})$ of $\mathfrak{h}$ is Noetherian \cite[Proposition 3.27]{Knapp1}, $\ker D_{\mathfrak{g},\mathfrak{h}}(M)$ must be finitely generated so that $H_D^{\mathfrak{g},\mathfrak{h}}(M)$ is finitely generated too.
\end{proof}

For an $\mathfrak{h}$-module $N$ belonging to $\mathcal{O}(\mathfrak{h})$ and a character $\xi_\mu:Z(\mathfrak{h})\rightarrow \mathbb{C}$, $\mu\in\mathfrak{t}^*$, of the center $Z(\mathfrak{h})$ of the universal enveloping algebra $U(\mathfrak{h})$ of $\mathfrak{h}$, analogously to \eqref{charr}, we set 
\begin{equation*}
	N^{\xi_\mu}:=\{v\in N\mid \forall z\in Z(\mathfrak{h}), \exists \hspace{1mm}n:=n(z,v)>0 \text{ s.t. }(z-\xi_\mu(z))^nv=0\}.
\end{equation*}
Then
\begin{equation*}
	\mathcal{O}(\mathfrak{h})=\bigoplus\limits_{\mu}\mathcal{O}(\mathfrak{h})_{\xi_\mu}
\end{equation*}
where $\mathcal{O}(\mathfrak{h})_{\xi_\mu}$ is the full subcategory of $\mathcal{O}(\mathfrak{h})$ consisting of modules $N\in\mathcal{O}(\mathfrak{h})$ such that $N^{\xi_\mu}=N$ \cite{humO}. Recall that $W$ stands for the Weyl group of $\mathfrak{g}$.

\begin{corollary}[Vogan's "conjecture"]\label{vogy} Let $M\in\mathcal{O}_{\chi_\lambda}$\hspace{1mm} for some $\lambda\in\mathfrak{t}^*$. Then $H_D^{\mathfrak{g},\mathfrak{h}}(M)^{\xi_\mu}\neq\{0\}$ only if the element $\mu\in\mathfrak{t}^*$ is $W$-conjugate to $\lambda$.
\end{corollary}

\begin{proof}
 The proof can be found in \cite{K3,pandzic,DH}. Namely, Huang and Pand\v zi\'c in the case of $(\mathfrak{g}, K)$-modules \cite{pandzic} and Kostant \cite{K3} in
 the more general case of pairs $(\mathfrak{g},\mathfrak{h})$ that we consider here show that there is a unital algebra morphism $\zeta:Z(\mathfrak{g})\rightarrow Z(\mathfrak{h}_\Delta)$ such that for every $z\in Z(\mathfrak{g})$ there exists an element $a\in [U(\mathfrak{g})\otimes Cl(\mathfrak{q})]^\mathfrak{h}$ such that 
 \begin{equation}\label{vcon}
 	z\otimes 1=\zeta(z)+D_{\mathfrak{g},\mathfrak{h}}a+aD_{\mathfrak{g},\mathfrak{h}}.
 \end{equation}
We note that here $D_{\mathfrak{g},\mathfrak{h}}$ is viewed as an element of the associative algebra $[U(\mathfrak{g})\otimes Cl(\mathfrak{q})]^\mathfrak{h}$.
 Moreover, it is shown that the map $\zeta$ satisfies the commutative diagram
 	\begin{equation*}
 	\begin{tikzcd}[cramped, sep=huge]
 	Z(\mathfrak{g})\arrow[d,"\zeta"]\arrow[r,"\text{H-C map}"]
 		& U(\mathfrak{t})^W\arrow[d,hook,]\\
 		Z(\mathfrak{h}_\Delta)\arrow[r,"\text{H-C map}"]
 		&U(\mathfrak{t})^{W_\mathfrak{h}}.
 	\end{tikzcd}
 \end{equation*} 
\vspace{2mm}

Hence, if $x\in \ker D_{\mathfrak{g},\mathfrak{h}}(M)$ and $x\notin \mathrm{im}\hspace{2mm}D_{\mathfrak{g},\mathfrak{h}}(M)$, $x$ determines a nonzero element $\tilde{x}$ in $H_D^{\mathfrak{g},\mathfrak{h}}(M)$ and we may assume that $\tilde{x}$ belongs to $H_D^{\mathfrak{g},\mathfrak{h}}(M)^{\xi_\mu}$ for some $\mu\in\mathfrak{t}^*$. 
Then, $aD_{g,h}(M)$ acts trivially on $x$ while, 
since $D_{\mathfrak{g},\mathfrak{h}}^2$
is central in
$[U(\mathfrak{g})\otimes Cl(\mathfrak{q})]^\mathfrak{h}$, we obtain
\begin{equation*}
D_{\mathfrak{g},\mathfrak{h}}(M)\big(D_{\mathfrak{g},\mathfrak{h}}(M)ax\big)=D_{\mathfrak{g},\mathfrak{h}}(M)^2ax=aD_{\mathfrak{g},\mathfrak{h}}(M)^2x=0.
\end{equation*}
so that $D_{\mathfrak{g},\mathfrak{h}}(M)ax \in \ker D_{\mathfrak{g},\mathfrak{h}}(M)\cap \mathrm{im}\hspace{1mm}D_{\mathfrak{g},\mathfrak{h}}(M)$. On the other hand, from \eqref{vcon}, one obtains
\begin{align}\label{katokato}
	\begin{split}
z\otimes 1-\chi_\lambda(z)=&\big(\zeta(z)-\xi_\mu(\zeta(z))\big)+\big(D_{\mathfrak{g},\mathfrak{h}}(M)a+aD_{\mathfrak{g},\mathfrak{h}}(M)\big)+\\
&\big(\xi_\mu(\zeta(z))-\chi_\lambda(z)\big).
\end{split}
\end{align}
We apply the both sides of this identity to $x$ certain times and we obtain
\begin{equation*}
	\big(\xi_\mu(\zeta(z))-\chi_\lambda(z)\big)x\in \ker D_{\mathfrak{g},\mathfrak{h}}(M)\cap \mathrm{im}\hspace{1mm} D_{\mathfrak{g},\mathfrak{h}}(M)
\end{equation*}
so that 
\begin{equation*}
	\big(\xi_\mu(\zeta(z))-\chi_\lambda(z)\big)\tilde{x}=0
\end{equation*}
and
\begin{equation*}
	\xi_\mu(\zeta(z))=\chi_\lambda(z).
	\end{equation*}
Then, from the above commutative diagram, one obtains that $\chi_\lambda=\chi_\mu$ so that $\mu$ is $W$-conjugate to $\lambda$ \cite{DH}.
\end{proof}

    In the last part of this section, we prove the following nonvanishing result for $H_D^{\mathfrak{g},\mathfrak{h}}(M)$.
    
\begin{theorem}\label{novanish}
	 Let $M$ be a nonzero $\mathfrak{g}$-module in $\mathcal{O}$. Then
	\begin{equation*}
	H_D^{\mathfrak{g},\mathfrak{h}}(M)\neq 0.
	\end{equation*}
	\end{theorem}

\begin{proof}
	Let $M$ be a nonzero $\mathfrak{g}$-module in $\mathcal{O}$. Let $\lambda\in\mathfrak{t}^*$ be the maximal weight of $M$ and $v_\lambda^+$ a nonzero maximal weight vector of the corresponding weight subspace $M_\lambda$. It suffices to show that 
 the vector $v_\lambda^+\otimes \mathbf{1}$ belongs to $H_D^{\mathfrak{g},\mathfrak{h}}(M)$.

First, we show that $v_\lambda^+\otimes \mathbf{1}$ is annihilated by $D_{\mathfrak{g},\mathfrak{h}}(M)$. Let $\{e_\alpha\}_{\alpha\in\Delta_\mathfrak{q}}$ be a set of root vectors belonging to $\mathfrak{q}$ and such that 
\begin{equation*}
\langle e_\alpha,e_{-\alpha}\rangle=1 \text{ for every }\alpha\in\Delta_\mathfrak{q}.
\end{equation*}
For every $\alpha\in\Delta^+_\mathfrak{q}$, set
\begin{equation*}
Z_\alpha^{(1)}:=\frac{\sqrt{2}}{2}(e_\alpha+e_{-\alpha})
\end{equation*}
and
\begin{equation*}
Z_\alpha^{(2)}:=\frac{\sqrt{2}}{2i}(e_\alpha-e_{-\alpha}).
\end{equation*}
Then
$\{Z_\alpha^{(i)}\mid\alpha\in\Delta^+_\mathfrak{q},i=1,2\}$ is an orthonormal basis of $\mathfrak{q}$.
Replacing this basis in \eqref{cubicDirac} gives
\begin{equation*}
D_{\mathfrak{g},\mathfrak{h}}(M)= D_++D_--1\otimes \gamma(c)
\end{equation*}
where 
\begin{equation*}
D_\pm:=\sum\limits_{\alpha\in\Delta_\mathfrak{q}^+}\pi(e_{\pm\alpha})\otimes \gamma(e_{\mp\alpha})
\end{equation*}
and $\gamma(c)$ is the cubic term \eqref{cubictermm}.
Then, for every $\alpha\in\Delta_\mathfrak{q}^+$, $\pi(e_\alpha)$ annihilates $v_\lambda^+$ while $\gamma(e_\alpha)$ annihilates $\mathbf{1}$ so that 
\begin{equation*}
D_\pm(v_\lambda\otimes \mathbf{1})=0.
\end{equation*}
On the other hand, the cubic term $\gamma(c)$ annihilates $\mathbf{1}$. More precisely, if $V_0$ stands for the $1$-dimensional simple $\mathfrak{g}$-module of highest weight $0$, replacing $V_0$ into \eqref{kernel} gives
\begin{equation*}
\ker \gamma(c)\cong\ker D_{\mathfrak{g},\mathfrak{h}}(V_0)\cong\bigoplus_{w\in W^1} F_{w\rho-\rho_{\mathfrak{h}}},
\end{equation*}
where 
 \begin{equation*}
W^1:=\{w\in W\mid \Delta_\mathfrak{h}^+\subseteq w\Delta^+\}
\end{equation*}
and $F_{\mu}$ stands for the simple $\mathfrak{h}$-module of highest weight $\mu\in\mathfrak{t}^*$. For $w=1$, $F_{\rho-\rho_\mathfrak{h}}$ is contained in $\ker \gamma(c)$ while the weight subspace $S_{\rho-\rho_\mathfrak{h}}$ of $S$ is $1$-dimensional and generated by $\mathbf{1}$. Therefore, $\mathbf{1}\in\ker \gamma(c)$ and we conclude that $v_\lambda^+\otimes \mathbf{1}\in\ker D_{\mathfrak{g},\mathfrak{h}}(M)$.

Now we show that $v_\lambda^+\otimes \mathbf{1}\notin \mathrm{im}\hspace{0.5mm}D_{\mathfrak{g},\mathfrak{h}}(M)$. Indeed,  assume that the weight vector $v_\lambda^+\otimes \mathbf{1}\in(M\otimes S)_{\lambda+\rho-\rho_\mathfrak{h}}$  belongs to $\mathrm{im}\hspace{1mm}D_{\mathfrak{g},\mathfrak{h}}(M)$. Then, there must be a weight vector $u\in(M\otimes S)_{\lambda+\rho-\rho_\mathfrak{h}}$ such that 
\begin{equation*}
D_{\mathfrak{g},\mathfrak{h}}(M)u=v_\lambda^+\otimes\mathbf{1}.
\end{equation*}
But $\lambda$ and $\rho-\rho_\mathfrak{h}$ are maximal weights of $M$ and $S$ respectively, so that $u$ is a linear combination of tensor products $v_i\otimes \mathbf{1}$ with each $v_i$ being a nonzero maximal weight vector of $M_\lambda$. Replacing $v_\lambda^+$ by $v_i$ in the above argument, we deduce that $D_{\mathfrak{g},\mathfrak{h}}(M)$ annihilates every element $v_i\otimes \mathbf{1}$. Therefore, 
\begin{equation*}
0=D_{\mathfrak{g},\mathfrak{h}}(M)u\neq v_\lambda^+\otimes\mathbf{1}
\end{equation*}
and
$v_\lambda^+\otimes \mathbf{1}\notin \mathrm{im}\hspace{0.5mm}D_{\mathfrak{g},\mathfrak{h}}(M)$. Hence $v_\lambda^+\otimes \mathbf{1}\in H_D^{\mathfrak{g},\mathfrak{h}}(M)$ and $H_D^{\mathfrak{g},\mathfrak{h}}(M)$ is nonzero.
\end{proof}

\section{Application to simple Verma modules}\label{Application}
Let $\mathfrak{g}$ be a complex semisimple Lie algebra and $\mathfrak{h}$ a Lie subalgebra of $\mathfrak{g}$ satisfying \eqref{conditionh}. Moreover, assume that $\mathfrak{h}$ contains a Cartan subalgebra $\mathfrak{t}$ of $\mathfrak{g}$.
In this section, we provide an application for simple Verma modules using results from Section \ref{Dirac cohomology and BGG}. More precisely, we show that the Dirac cohomology  of a simple Verma $\mathfrak{g}$-module turns out to be a simple Verma $\mathfrak{h}$-module.

%	 Let $\mathfrak{h}$ be a subalgebra of $\mathfrak{g}$ satisfying \eqref{conditionh} and consider the Dirac operator 
%	\begin{equation*}
%	D_{\mathfrak{g},\mathfrak{h}}(M(\lambda)):M(\lambda)\otimes S\rightarrow M(\lambda)\otimes S.
%	\end{equation*}
	
	\begin{theorem}\label{simple}
		Let $M(\lambda)$ be the simple Verma module of highest weight $\lambda\in\mathfrak{t}^*$. Then
		\begin{equation*}
		H_D(M(\lambda))\cong M_\mathfrak{h}(\lambda+\rho-\rho_\mathfrak{h}),
		\end{equation*}
		where $M_\mathfrak{h}(\lambda+\rho-\rho_\mathfrak{h})$ stands for the simple Verma $\mathfrak{h}$-module of highest weight $\lambda+\rho-\rho_\mathfrak{h}$.
		\end{theorem}
	
	\begin{proof}	Let $M(\lambda)$, $\lambda\in\mathfrak{t}^*$, be the Verma module with highest weight $\lambda\in\mathfrak{t}^*$, and suppose that $M(\lambda)$ is simple. Then $\lambda$ is $\rho$-antidominant \cite[Section 4.8]{humO}, i.e. 
		\begin{equation*}
		2\dfrac{\langle\lambda+\rho,\alpha\rangle}{\lVert\alpha\rVert^2}\notin \mathbb{Z}_{>0}\text{ for every }\alpha\in\Delta^+.
		\end{equation*}
		
	According to Theorem \ref{voganO}, $H_D^{\mathfrak{g},\mathfrak{h}}(M(\lambda))$ belongs to $\mathcal{O}(\mathfrak{h})$ so that it is of finite length, while, from Theorem \ref{novanish}, $H_D^{\mathfrak{g},\mathfrak{h}}(M(\lambda))$ is nonzero. Let the simple $\mathfrak{h}$-module $L_\mathfrak{h}(\mu)$ of highest weight $\mu\in\mathfrak{t}^*$ be a composition factor of $H_D^{\mathfrak{g},\mathfrak{h}}(M(\lambda))$. According to Corollary \ref{vogy}, $\mu+\rho_\mathfrak{h}$ must belong to the $W$-orbit of $\lambda+\rho$, i.e. $\mu=w(\lambda+\rho)-\rho_\mathfrak{h}$\hspace{1mm} for some $w\in W$. Since $\lambda$ is $\rho$-antidominant, $\lambda+\rho\leq w(\lambda+\rho)$ \cite[Proposition 3.5]{humO} so that $\mu=w(\lambda+\rho)-\rho_\mathfrak{h}\geq\lambda+\rho-\rho_\mathfrak{h}$. But $\lambda+\rho-\rho_\mathfrak{h}$ is the maximal weight of $M(\lambda)\otimes S$ and one concludes that $\mu=\lambda+\rho-\rho_\mathfrak{h}$. Therefore, every composition factor of $H_D^{\mathfrak{g},\mathfrak{h}}(M(\lambda))$ is isomorphic to  $L_\mathfrak{h}(\lambda+\rho-\rho_\mathfrak{h})$. 
	
	On the other hand,
	the weight subspace $(M\otimes S)_{\lambda+\rho-\rho_\mathfrak{h}}$ of $M\otimes S$ is generated by a nonzero highest weight vector $v_\lambda^+$ of $M(\lambda)$ tensorized with the weight vector $\mathbf{1}$ of $S$, and is $1$-dimensional. Therefore, $L_\mathfrak{h}(\lambda+\rho-\rho_\mathfrak{h})$ appears only once in $H_D^{\mathfrak{g},\mathfrak{h}}(M(\lambda)$ so that $H_D^{\mathfrak{g},\mathfrak{h}}(M(\lambda))\cong L_\mathfrak{h}(\lambda+\rho-\rho_\mathfrak{h})$. The equality $L_\mathfrak{h}(\lambda+\rho-\rho_\mathfrak{h})=M_\mathfrak{h}(\lambda+\rho-\rho_\mathfrak{h})$ follows from the fact that the element $\lambda+\rho-\rho_\mathfrak{h}$ turns out to be $\rho_\mathfrak{h}$-antidominant so that $M_\mathfrak{h}(\lambda+\rho-\rho_\mathfrak{h})$ is simple \cite[Section 4.8]{humO}.

\end{proof}

  \section{Dirac and Lie algebra cohomology for $\mathcal{O}$}\label{DiracandLie}
  
  In this section, we extend the method used in \cite{renardpandzic}, in order to relate Dirac cohomology of unitary $\mathfrak{g}$-modules $M\in\mathcal{O}$ and nilpotent Lie algebra cohomology with coefficients in $M$. 
  
  Fix $\mathfrak{g}_0$ to be a $\theta$-stable real form of $\mathfrak{g}$ so that $\mathfrak{g}_0=\mathfrak{k}_0\stackrel{\theta}{\oplus}\mathfrak{p}_0$ is the corresponding Cartan decomposition.
  
  \begin{definition}
  	A $\mathfrak{g}$-module $(M,\pi)\in\mathcal{O}$ is said to be unitary if it admits a Hermitian inner product $\langle\cdot,\cdot\rangle$ such that 
  	\begin{equation*}
  	\langle \pi(X)u,v\rangle+\langle u,\pi(X)v\rangle=0
  	\end{equation*}
  	for every $u,v\in M$ and every $X\in\mathfrak{g}_0$.
  \end{definition}
In other words, $\pi(X)^\text{adj}=-\pi(X)$ for every $X\in\mathfrak{g}_0$. One can easily check that, if \hspace{1mm}$\widetilde{\cdot}$ stands for the conjugation with respect to $\mathfrak{g}_0$, then 
\begin{equation}\label{adjointO}
\pi(X)^{\text
	{adj}}=-\pi(\widetilde{X}) \text{ for every }X\in\mathfrak{g}.
\end{equation}
For more details concerning the classification of unitary highest weight modules, see \cite{enright} for the case when $\mathfrak{g}_0$ is a compact real form of $\mathfrak{g}$, and \cite{noncompactform} for the case when $\mathfrak{g}_0$ is a noncompact real form of $\mathfrak{g}$.

Assume now that $\mathfrak{g}$ and $\mathfrak{k}:=(\mathfrak{k}_0)^\mathbb{C}$ form a Hermitian symmetric pair. Then, %$\mathfrak{p}:=\mathfrak{k}^\perp$ is an abelian Lie algebra, 
$\text{rank}_{\mathbb{C}}(\mathfrak{g})=\text{rank}_\mathbb{C}(\mathfrak{k})$ and $\mathfrak{g}=\mathfrak{k}\oplus\mathfrak{p}^+\oplus\mathfrak{p}^-$, with $\mathfrak{p}^+$ and $\mathfrak{p}^-$ being maximal isotropic dual subspaces of $\mathfrak{p}:=(\mathfrak{p}_0)^\mathbb{C}$ with respect to the restriction of the Killing form on $\mathfrak{p}$. Moreover $\mathfrak{p}$ is abelian and $\mathfrak{k}\oplus\mathfrak{p}^+$ is a parabolic subalgebra of $\mathfrak{g}$ while we make the assumption that $\mathfrak{k}\oplus \mathfrak{p}^+\supseteq\mathfrak{b}$. 

Let $M$ be a $\mathfrak{g}$-module. The spaces $ \mathrm{Hom}(\bigw^k \mathfrak{p}^+,M)$, $k\geq0$, provide the Chevalley-Eilenberg complex which calculates the nilpotent Lie algebra cohomology of $\mathfrak{p}^+$ with coefficients in $M$. The differential
\begin{equation*}
	d^k: \mathrm{Hom}(\bigw\hspace{-.5mm}{}^k \hspace{1mm}\mathfrak{p}^+,M)\rightarrow \mathrm{Hom}(\bigw\hspace{-.5mm}{}^{k+1} \hspace{1mm}\mathfrak{p}^+,M)
\end{equation*}
of this complex is given by 
\begin{align*}
	d^k\omega(X_0\wedge&\ldots\wedge X_k):=\sum\limits_{i=0}^k (-1)^iX_i\big(\omega(X_0\wedge\ldots \wedge\widehat{X}_i\wedge\ldots \wedge X_k)\big),
\end{align*}
where $X_i\in\mathfrak{p}^+$, $\omega\in\mathrm{Hom}(\bigw\hspace{-.5mm}{}^k \hspace{1mm}\mathfrak{p}^+,M)$ and $\widehat{\cdot}$ indicates that we omit the corresponding element. 
Then, the $k$-th nilpotent Lie algebra cohomology $H^k(\mathfrak{p}^+,M)$ of $\mathfrak{p}^+$ with coefficients in $M$ is 
\begin{equation*}
	H^k(\mathfrak{p}^+,M)=\frac{\ker d^k}{\mathrm{im}\hspace{0.3mm}d^{k-1}}.
\end{equation*}
We note that the map $d^k$ is $\mathfrak{k}$-equivariant so that $H^k(\mathfrak{p}^+,M)$ is a $\mathfrak{k}$-module for every $k\geq0$.

On the other hand, the spaces $M\otimes \bigw^{\hspace{-0.3mm}k} \mathfrak{p}^-$, $k\geq0$, provide the complex for the nilpotent Lie algebra homology $H_k(\mathfrak{p}^-,M)$ of $\mathfrak{p}^-$ with coefficients in $M$ where the differential
\begin{equation*}
	\partial^k:M\otimes \bigw\hspace{-.5mm}{}^{k+1} \hspace{1mm}\mathfrak{p}^-\rightarrow M\otimes \bigw\hspace{-.5mm}{}^{k} \hspace{1mm}\mathfrak{p}^-
\end{equation*}
is given by
\begin{align*}
	\partial^k\big(v\otimes Y_1\wedge&\ldots\wedge Y_{k+1}\big):=\sum\limits_{i=1}^{k+1} (-1)^{i+1}Y_iv\otimes Y_1\wedge\ldots \wedge\widehat{Y}_i\wedge\ldots \wedge Y_{k+1}.
\end{align*}
The $k$-th nilpotent Lie algebra homology $H_k(\mathfrak{p}^-,M)$ of $\mathfrak{p}^-$ with coefficients in $M$ is
\begin{equation*}
	H_k(\mathfrak{p}^-,M)=\frac{\ker \partial^{k-1}}{\mathrm{im} \partial^{k}}.
\end{equation*}
As before, the map $\partial^k$ is $\mathfrak{k}$-equivariant so that $H_k(\mathfrak{p}^-,M)$ is a $\mathfrak{k}$-module for every $k\geq0$.
 
\begin{theorem}\label{cohom}
	Let $\mathfrak{g},\mathfrak{k}$ be a Hermitian symmetric pair and $M$ a unitary module belonging to $\mathcal{O}$. Then 
	\begin{equation*}
	H_D^{\mathfrak{g},\mathfrak{k}}(M)\cong \ker D_{\mathfrak{g},\mathfrak{k}}(M)\cong H^\cdot(\mathfrak{p}^+,M)\otimes \mathbb{C}_{\rho-\rho_\mathfrak{k}}\cong H_\cdot(\mathfrak{p}^-,M)\otimes \mathbb{C}_{\rho-\rho_\mathfrak{k}},
	\end{equation*}
	where $H^\cdot(\mathfrak{p}^+,M)$ (respectively $H_\cdot(\mathfrak{p}^-,M)$) stands for the nilpotent Lie algebra cohomology (respectively homology) with coefficients in $M$.
	%In other words, $H_D^{\mathfrak{g},\mathfrak{h}}(M)$ coincides, up to some character, with the nilpotent Lie algebra (co)homology with coefficients in $M$.
\end{theorem}

In what follows,  in order to prove Theorem \ref{cohom}, we provide a Hodge decomposition for the operator $D_{\mathfrak{g},\mathfrak{k}}(M)$.
Let $\{Z_i\}$ be an orthonormal basis of $\mathfrak{p}_0$ and set
\begin{equation}\label{basesduales}
e_j:=\frac{Z_{2j-1}+iZ_{2j}}{\sqrt{2}},\quad e_{-j}:=\frac{Z_{2j-1}-iZ_{2j}}{\sqrt{2}}.
\end{equation}
Then $\{e_j\}$ and $\{e_{-j}\}$ are dual bases of the maximal isotropic dual subspaces $\mathfrak{p}^+:=\text{span}\{e_j\}_j$ and $\mathfrak{p}^-:=\text{span}\{e_{-j}\}_j$ of $\mathfrak{p}$, respectively. Define operators
 
\begin{align*}
C^+&:=\sum\limits_{j}\pi(e_j)\otimes \gamma(e_{-j})\\
C^-&:=\sum\limits_{j}\pi(e_{-j})\otimes \gamma(e_{j})
\end{align*}
so that $C^++C^-=D_{\mathfrak{g},\mathfrak{k}}(M)$.

Since $\mathfrak{p}^+$ and $\mathfrak{p}^-$ are maximal isotropic dual subspaces of $\mathfrak{p}$, 
\begin{equation*}S\cong \bigw \mathfrak{p}^-\otimes \mathbb{C}_{\rho-\rho_\mathfrak{k}}
\end{equation*}
 and
\begin{align}\label{identification}
\begin{split}
M\otimes S&\cong M\otimes \bigw \mathfrak{p}^-\otimes\mathbb{C}_{\rho-\rho_\mathfrak{k}}\cong M\otimes \bigw(\mathfrak{p}^+)^*\otimes \mathbb{C}_{\rho-\rho_\mathfrak{k}}\\
&
\cong \mathrm{Hom}(\bigw \mathfrak{p}^+,M)\otimes \mathbb{C}_{\rho-\rho_\mathfrak{k}}.\end{split}\end{align}

\begin{lemma}\label{identify}
	Under the identification \eqref{identification}, 
	$C^+=d$ and  $C^-=\partial $ so that 
	\begin{equation*}
	D_{\mathfrak{g},\mathfrak{k}}(M)=d+\partial.
	\end{equation*}
	\end{lemma}

\begin{proof}See \cite[Section 9.1.4]{pandzic}.
	\end{proof}

\begin{lemma} Let $M$ be a unitary $\mathfrak{g}$-module. Then, there is a Hermitian inner product $\langle\cdot,\cdot\rangle$ on $M\otimes S$ such that 
	\begin{equation*}
	(C^\pm)^\text{adj}=-C^{\mp}.
	\end{equation*}
	Hence $D_{\mathfrak{g},\mathfrak{k}}(M)$ is skew-adjoint.
	\end{lemma}

\begin{proof} According to \cite{thesis,pandzic}, $S=\bigw \mathfrak{p}^-$ admits a Hermitian inner product $\langle\cdot,\cdot\rangle_S$ and
	\begin{equation*}
	\gamma(X)^\text{adj}=\gamma(\widetilde{X}) \text{ for every } X\in\mathfrak{p},
	\end{equation*}
    where $\widetilde{\cdot}$ stands for the conjugation with respect to $\mathfrak{g}_0$.
	On the other hand, by assumption and \eqref{adjointO}, there is a Hermitian inner product $\langle\cdot,\cdot\rangle_M$ on $M$ such that 
	\begin{equation*}
	\pi(X)^\text{adj}=-\pi(\widetilde{X})\text{ for every }X\in\mathfrak{p}.
	\end{equation*}
	On $M\otimes S$ define the natural inner product $\langle\cdot,\cdot\rangle$ coming from tensorizing $\langle\cdot,\cdot\rangle_S$ and $\langle\cdot,\cdot\rangle_M$, i.e.
	\begin{equation*}
	\langle v_1\otimes u_1,v_2\otimes u_2\rangle:=\langle v_1,v_2\rangle_M\langle u_1,u_2\rangle_S
	\end{equation*}
		for every $v_1\otimes u_1$, $v_2\otimes u_2\in M\otimes S$.
		Then, if $e_j$ and $e_{-j}$ are as in \eqref{basesduales}, we have 
		\begin{align*}
		\big(\pi(e_j)\otimes \gamma(e_{-j})\big)^\text{adj}&=\pi(e_j)^\text{adj}\otimes \gamma(e_{-j})^\text{adj}\\
		&=-\pi(e_{-j})\otimes \gamma(e_{j})
		\end{align*}
		so that $(C^+)^\text{adj}=-C^-$, $(C^-)^\text{adj}=-C^+$ and
	
		\begin{equation*}
		D_{\mathfrak{g},\mathfrak{k}}(M)^\text{adj}=-D_{\mathfrak{g},\mathfrak{k}}(M).
		\end{equation*}
\end{proof}

\begin{lemma}\label{dcmp}
	Let $(\mathfrak{g},\mathfrak{k})$ be a Hermitian symmetric pair and $M$ a unitary module in $\mathcal{O}$. Then 
	
	\begin{equation*}M\otimes S =\ker D_{\mathfrak{g},\mathfrak{k}}(M)\oplus \mathrm{im}\hspace{0.5mm}D_{\mathfrak{g},\mathfrak{k}}(M).
	\end{equation*}
\end{lemma}

\begin{proof}
According to Corollary \ref{decom}, 
\begin{equation*}
M\otimes S=\big(M\otimes S\big)(0)\oplus\bigoplus\limits_{c\neq0}\big(M\otimes S\big)(c),
\end{equation*}
where $\big(M\otimes S\big)(c)$, $c\in\mathbb{R}$, stands for the generalized eigenspace of $D_{\mathfrak{g},\mathfrak{k}}(M)^2$ of eigenvalue $c$. It suffices to show that 
\begin{align*}
\ker D_{\mathfrak{g},\mathfrak{k}}(M)&=\big(M\otimes S\big)(0)\\
\text{ and }\mathrm{im}\hspace{0.5mm}D_{\mathfrak{g},\mathfrak{k}}(M)&=\bigoplus\limits_{c\neq0}\big(M\otimes S\big)(c).
\end{align*}

For the first equality, since $D_{\mathfrak{g},\mathfrak{k}}(M)$ is skew-adjoint, $\ker D_{\mathfrak{g},\mathfrak{k}}(M)^2=\ker D_{\mathfrak{g},\mathfrak{k}}(M)$. Therefore, one easily checks that $\ker D_{\mathfrak{g},\mathfrak{k}}(M)^k=\ker D_{\mathfrak{g},\mathfrak{k}}(M)$ for every $k>0$ so that $\ker D_{\mathfrak{g},\mathfrak{k}}(M)=\big(M\otimes S\big)(0)$.

For the second equality, using again the fact that $D_{\mathfrak{g},\mathfrak{k}}(M)$ is skew-adjoint and the first equality, we obtain
\begin{align*}
\mathrm{im}\hspace{0.5mm}D_{\mathfrak{g},\mathfrak{k}}(M)&=\ker D_{\mathfrak{g},\mathfrak{k}}(M)^\perp\\
&=\big(M\otimes S\big)(0)^\perp\\
&=\bigoplus\limits_{c\neq0}\big(M\otimes S\big)(c)
\end{align*}
	\end{proof}

\begin{proposition}\label{manyresults}
	Let $(\mathfrak{g},\mathfrak{k})$ be a Hermitian symmetric pair and $M$ a unitary module in $\mathcal{O}$. Then 
	\begin{itemize}
		\item[(i)] $\ker C^+=\mathrm{im}\hspace{0.5mm}C^+ \oplus\ker D_{\mathfrak{g},\mathfrak{k}}(M)$
		\item[(ii)] $\ker C^-=\mathrm{im}\hspace{0.5mm}C^- \oplus\ker D_{\mathfrak{g},\mathfrak{k}}(M)$.
	\end{itemize}
The above direct sums are orthogonal with respect to $\langle\cdot,\cdot\rangle$.
\end{proposition}

\begin{proof} The proof is similar to the proof in the case of a unitary $(\mathfrak{g},K)$-modules \cite[Lemma 9.3.3 \& Proposition 9.3.4]{pandzic}.
\end{proof}

Consequently, we obtain
\begin{equation*}
H_D^{\mathfrak{g},\mathfrak{k}}(M)\cong \ker D_{\mathfrak{g},\mathfrak{k}}(M)\cong H^\cdot(\mathfrak{p}^+,M)\otimes \mathbb{C}_{\rho-\rho_\mathfrak{k}}\cong H_\cdot(\mathfrak{p}^-,M)\otimes \mathbb{C}_{\rho-\rho_\mathfrak{k}}
\end{equation*}
which is exactly the statement of Theorem \ref{cohom}. Here, the first equivalence is due to Lemma \ref{dcmp} while the second and the third ones due to Lemma \ref{identify} and Proposition \ref{manyresults}.

\section{Higher Dirac cohomology and index  \\ \hspace{1.5mm}for the BGG category $\mathcal{O}$}
Let $\mathfrak{g}$ be a complex semisimple Lie algebra and $\mathfrak{h}$ a Lie subalgebra of $\mathfrak{g}$ satisfying \eqref{conditionh}.
Let $\mathfrak{h}$ contain a Cartan subalgebra $\mathfrak{t}$ of $\mathfrak{g}$, so that the spin module $S$ has a splitting into $S=S^+\oplus S^-$. Let $X_1,X_2,X_3$ and $X$ be $\mathfrak{g}$-modules such that the restrictions $(X_1)_{\mid\mathfrak{h}}$, $(X_2)_{\mid\mathfrak{h}}$,$(X_3)_{\mid\mathfrak{h}}$ and $X_{\mid\mathfrak{h}}$ to $\mathfrak{h}$ decompose into algebraic finite-multiplicity direct sums of finite dimensional simple $\mathfrak{h}$-modules. Assume that $X_1,X_2,X_3$ fit into a short exact sequence
\begin{equation*}
0\longrightarrow X_1\longrightarrow X_2\longrightarrow X_3\longrightarrow 0.
\end{equation*}
In the case where each $X_1,X_2,X_3$, and $X$ has infinitesimal character (or it is a direct sum of modules having infinitesimal characters), we can construct $\mathfrak{h}$-morphisms so that the following diagram of Dirac cohomology spaces is exact \cite[Subsection 9.6.2]{pandzic}:

	\begin{equation}\label{exactcircle}
\begin{tikzcd}[cramped, sep=huge]
H_D^{\mathfrak{g},\mathfrak{h}}(X_1)^+\arrow[r] 
& H_D^{\mathfrak{g},\mathfrak{h}}(X_2)^+ \arrow[r] &H_D^{\mathfrak{g},\mathfrak{h}}(X_3)^+ \arrow[d]\\
H_D^{\mathfrak{g},\mathfrak{h}}(X_3)^-\arrow[u]
&\arrow[l]H_D^{\mathfrak{g},\mathfrak{h}}(X_2)^- &\arrow[l]H_D^{\mathfrak{g},\mathfrak{h}}(X_1)^-
\end{tikzcd}
\end{equation}
while for the Dirac index $I_D(X)$ of $X$, we have \cite{diracindex}:
\begin{equation}\label{higherrelat}
I_D(X)=X\otimes S^+-X\otimes S^-.
\end{equation}
%\begin{remark} The maps in \eqref{exactcircle} are not natural
%\end{remark}
Nevertheless, when $X_1$, $X_2$, $X_3$ and $X$, instead of infinitesimal characters, have generalized infinitesimal characters, the above statements may fail \cite[Section 2]{somberg}.

 In order to overcome this problem, Pand\v zi\'c and Somberg defined a notion of higher Dirac cohomology \cite{somberg}:

\begin{definition}
	For $\mathfrak{g}$, $\mathfrak{h}$ and $X$ as above, we call $k$-th higher Dirac cohomology of $X$ the $\mathfrak{h}$-module
	\begin{equation*} 
	H^k_{\mathrm{top}}(X):=\frac{\ker D_{\mathfrak{g},\mathfrak{h}}(X)^{2k+1}}{\ker D_{\mathfrak{g},\mathfrak{h}}(X)^{2k+1}\cap\mathrm{im}\hspace{0.5mm}D_{\mathfrak{g},\mathfrak{h}}(X)+\ker D_{\mathfrak{g},\mathfrak{h}}(X)^{2k}}
	\end{equation*}
	and higher Dirac cohomology of $X$ the $\mathfrak{h}$-module
	\begin{equation*}
	H_{\mathrm{top}}(X):=\bigoplus\limits_{k\geq 0}H_{\mathrm{top}}^k(X).
	\end{equation*}
\end{definition}
Using the above definition of higher Dirac cohomology, they introduced a notion of higher Dirac index:

\begin{definition}For $\mathfrak{g}$, $\mathfrak{h}$ and $X$ as above, the virtual $\mathfrak{h}$-module
	\begin{equation*}
	I_\mathrm{top}(X):=H_{\mathrm{top}}(X)^+-H_{\mathrm{top}}(X)^-
	\end{equation*}
 is called the higher Dirac index of $X$.
\end{definition}

\begin{remark}Note that in the case where $X$ has infinitesimal character, $H_D^{\mathfrak{g},\mathfrak{h}}(X)$ (respectively $I_D(X)$) and $H_{\mathrm{top}}(X)$ (respectively $I_\mathrm{top}(X)$) coincide. 
	\end{remark}

Then, they proved that if, instead of the standard notions of Dirac cohomology $H_D^{\mathfrak{g},\mathfrak{k}}(X)$ and index $I_D(X)$, we use the higher Dirac cohomology $H_\mathrm{top}(X)$ and index $I_\mathrm{top}(X)$, the statements for \eqref{exactcircle} and \eqref{higherrelat} remain true whenever $X_1,X_2,X_3$ and $X$ have generalized infinitesimal characters. The proofs of their results are based on the fact that the generalized kernel $\big(X\otimes S\big)(0)$ of $D_{\mathfrak{g},\mathfrak{h}}(X)$ is finite dimensional and can be decomposed into Jordan blocks for $D_{\mathfrak{g},\mathfrak{k}}(X)$. \begin{remark} For $X_1$, $X_2$ and $X_3$ having generalized infinitesimal characters, the maps used in the corresponding diagram \eqref{exactcircle} depend on the chosen Jordan decomposition and, so, are noncanonical.
\end{remark}
 
In this section, we prove similar results in the case where, instead of $\mathfrak{g}$-modules $X_1,X_2,X_3$ and $X$ satisfying the previous finiteness condition for their restrictions to $\mathfrak{h}$, we have $\mathfrak{g}$-modules $M_1,M_2,M_3$ and $M$ belonging to $\mathcal{O}$. Of course, we can again have a notion of higher Dirac cohomology $H_\mathrm{top}(M)$ and index $I_\mathrm{top}(M)$. Nevertheless, the generalized kernel $\big(M\otimes S)(0)$ of $D_{\mathfrak{g},\mathfrak{h}}(M)$ need not be finite dimensional and may not have a Jordan decomposition for $D_{\mathfrak{g},\mathfrak{h}}(M)$. In order to overcome this problem, we decompose $\big(M\otimes S\big)(0)$ into a direct sum of weight subspaces 
\begin{equation*}
\big(M\otimes S\big)(0)=\bigoplus\limits_{\mu\in\mathfrak{t}^*}	[\big(M\otimes S\big)(0)]_\mu,
\end{equation*}
where each weight subspace $[\big(M\otimes S\big)(0)]_\mu$ is finite dimensional. Since $\mathfrak{t}\subseteq\mathfrak{h}$ and $D_{\mathfrak{g},\mathfrak{h}}(M)$ is $\mathfrak{h}$-equivariant, $D_{\mathfrak{g},\mathfrak{h}}(M)$ is $\mathfrak{t}$-equivariant so that it preserves each weight subspace. Now, each weight subspace $[\big(M\otimes S\big)(0)]_\mu$ has a Jordan block decomposition for $D_{\mathfrak{g},\mathfrak{h}}(M)$ and we can apply the methods used in \cite{somberg} and construct maps so that

	\begin{equation}\label{exactcircleweight}
\begin{tikzcd}[cramped, sep=huge]
H_\mathrm{top}(M_1)^+_\mu\arrow[r] 
& H_\mathrm{top}(M_2)^+_\mu \arrow[r] &H_\mathrm{top}(M_3)^+_\mu \arrow[d]\\
H_\mathrm{top}(M_3)^-_\mu\arrow[u]
&\arrow[l]H_\mathrm{top}(M_2)^-_\mu &\arrow[l]H_\mathrm{top}(M_1)^-_\mu
\end{tikzcd}
\end{equation}
forms an exact circle of weight subspaces, i.e. of $\mathfrak{t}$-modules, while for the weight subspaces $I_{\mathrm{top}}(M)_\mu$ and $\big(M\otimes S^\pm\big)_\mu$ we have
\begin{equation}\label{higherrelatweight}
I_\mathrm{top}(M)_\mu=\big(M\otimes S^+\big)_\mu-\big(M\otimes S^-\big)_\mu.
\end{equation}
%Nevertheless, we have lost the $\mathfrak{h}$-module structure and it is not obvious how we can recover it. Of course, one could take the direct sums over $\mu\in\mathfrak{t}^*$ but again the diagram which we will construct will be a diagram of $\mathfrak{t}$-modules while we are interested in the $\mathfrak{h}$-module structure of each space. 
%
%In the proofs of the following results, we show how one, having used the finite dimensionality of the weight subspaces, can, then, recover the $\mathfrak{h}$-module structure of each space.

We start with the result concerning the higher Dirac index which is the analogue of \cite[Theorem 1.11]{somberg}.
%prove the statements \eqref{exactcircle} and \eqref{higherrelat} for the weight subspaces $H_\mathrm{top}(M_i)_\mu$, where $M_i\in\mathcal{O}$ for $i=1,2,3$.

\begin{proposition}\label{propindex}
	Let $M\in\mathcal{O}$. Then
	\begin{equation*}
	I_{\mathrm{top}}(M)=M\otimes S^+-M\otimes S^- 
	\end{equation*}
\end{proposition}

\begin{proof}[Proof of Proposition \ref{propindex}]
	Let $\big(M\otimes S\big)(c)$ be the generalized eigenspace of $D_{\mathfrak{g},\mathfrak{h}}(M)^2$ with eigenvalue $c\in\mathbb{R}$. Since $D_{\mathfrak{g},\mathfrak{h}}(M)^2$ is even in the Clifford factor, $\big(M\otimes S\big)(c)$ decomposes into
	\begin{equation*}
	\big(M\otimes S\big)(c):=\big(M\otimes S\big)(c)^+\oplus\big(M\otimes S\big)(c)^-.
	\end{equation*}
	Moreover, since $D_{\mathfrak{g},\mathfrak{h}}(M)$ is $\mathfrak{t}$-equivariant, one obtains a more refined decomposition in terms of weight subspaces:
		\begin{equation*}
	[\big(M\otimes S\big)(c)]_\mu:=[\big(M\otimes S\big)(c)^+]_\mu\oplus[\big(M\otimes S\big)(c)^-]_\mu.
	\end{equation*}
	
	If $c\neq0$, $D_{\mathfrak{g},\mathfrak{h}}(M)$ maps injectively $[\big(M\otimes S\big)(c)^\pm]_\mu$ into $[\big(M\otimes S\big)(c)^\mp]_\mu$. More precisely, if $x\in[\big(M\otimes S\big)(c)^\pm]_\mu$ and $D_{\mathfrak{g},\mathfrak{h}}(M)x=0$, then $D_{\mathfrak{g},\mathfrak{h}}(M)^2x=0$ and so $x$ must be $0$.
	Since $[\big(M\otimes S\big)(c)^\pm]_\mu$ is finite dimensional, one deduces that the restriction 
	\begin{equation*}
	D_{\mathfrak{g},\mathfrak{h}}(M):[\big(M\otimes S\big)(c)^\pm]_\mu\rightarrow[\big(M\otimes S\big)(c)^\mp]_\mu
	\end{equation*}
	is an isomorphism. Hence, $\big(M\otimes S\big)(c)^+\cong\big(M\otimes S\big)(c)^-$ for every $c\neq0$ so that
	\begin{equation*}
	M\otimes S^+-M\otimes S^-=\big(M\otimes S\big)(0)^+-\big(M\otimes S\big)(0)^-.
	\end{equation*}
	The following lemma completes the proof.
	\end{proof}

	\begin{lemma}\label{lemdecomp} Let $M\in\mathcal{O}$. Then
		\begin{equation*}\big(M\otimes S\big)(0)^+-\big(M\otimes S\big)(0)^-=H_{\mathrm{top}}(M)^+-H_{\mathrm{top}}(M)^-.
		\end{equation*}
		\end{lemma}

		\begin{proof} Since we are working in the Grothendieck group of $\mathfrak{h}$, it suffices to prove the statement for the corresponding formal characters. In other words, it suffices to show that, for every $\mu\in\mathfrak{t}^*$,
			\begin{equation*}
			\dim [\big(M\otimes S\big)(0)^+]_\mu-\dim [\big(M\otimes S\big)(0)^-]_\mu=\dim H_{\mathrm{top}}(M)^+_\mu-H_{\mathrm{top}}(M)^-_\mu.
			\end{equation*}

	  For every $\mu\in\mathfrak{t}^*$, $[\big(M\otimes S\big)(0)]_\mu$ is finite dimensional and can be decomposed into Jordan blocks for $D_{\mathfrak{g},\mathfrak{h}}(M)$:
	  \begin{equation*}
	  [\big(M\otimes S\big)(0)]_\mu=J_1^{(\mu)}\oplus \ldots\oplus J_{n_\mu}^{(\mu)}.
	  \end{equation*}
	  For every $j\in\{1,\ldots,n_\mu\}$, the Jordan block $J^{(\mu)}_j$ can be decomposed into a direct sum of $1$-dimensional subspaces
	  	\begin{equation*}
	  W_1^{(\mu,j)}\oplus W_2^{(\mu,j)}\oplus \ldots\oplus W_{m_{\mu,j}}^{(\mu,j)}
	  \end{equation*}
	   such that
	  \begin{equation*}
	D_{\mathfrak{g},\mathfrak{h}}(M)W^{(\mu,j)}_1=0
	\end{equation*}
	 and 
	 \begin{equation*}
	 D_{\mathfrak{g},\mathfrak{h}}(M):W_{i}^{(\mu,j)}\xrightarrow{\cong}W_{i-1}^{(\mu,j)} \hspace{0.5mm}\text{ for } i=2,\ldots,m_{\mu,j}.
	 \end{equation*}
	   We say that the Jordan block $J^{(\mu)}_j$ is of size $m_{\mu,j}$. Moreover, we can choose each $W_i^{(\mu,j)}$ to be contained either in $\big(M\otimes S\big)(0)^+$ or in $\big(M\otimes S\big)(0)^-$. For every $k>0$, set 
	   \begin{equation*}
	   N_{\mu,k}:=\bigoplus\limits_{j=1}^{n_\mu}W_k^{(\mu,j)}\text{ and }N_{\mu,k}^\pm=N_{\mu,k}\cap\big(M\otimes S\big)^\pm.
	   \end{equation*}
   
   \begin{figure}
   \begin{tikzcd}[row sep=small]	
   	&N_{\mu,3} &  N_{\mu,2} & N_{\mu,1} &\\
   	J_1^{(\mu)}:\hspace{2mm}\ldots \arrow[r,"D_{\mathfrak{g},\mathfrak{h}}(M)"]&\fbox{\Centerstack[l]{$W_3^{(\mu,1)}$}} \arrow[r,"D_{\mathfrak{g},\mathfrak{h}}(M)"] &  \fbox{\Centerstack[l]{$W_2^{(\mu,1)}$}} \arrow[r,"D_{\mathfrak{g},\mathfrak{h}}(M)"]& \fbox{\Centerstack[l]{$W_1^{(\mu,1)}$}}\arrow[r,"D_{\mathfrak{g},\mathfrak{h}}(M)"]&0\\
   	J_2^{(\mu)}:\hspace{2mm}\ldots \arrow[r,"D_{\mathfrak{g},\mathfrak{h}}(M)"]&\fbox{\Centerstack[l]{$W_3^{(\mu,2)}$}} \arrow[r,"D_{\mathfrak{g},\mathfrak{h}}(M)"] &  \fbox{\Centerstack[l]{$W_2^{(\mu,2)}$}} \arrow[r,"D_{\mathfrak{g},\mathfrak{h}}(M)"]& \fbox{\Centerstack[l]{$W_1^{(\mu,2)}$}}\arrow[r,"D_{\mathfrak{g},\mathfrak{h}}(M)"]&0\\
   	J_3^{(\mu)}:\hspace{2mm}\ldots \arrow[r,"D_{\mathfrak{g},\mathfrak{h}}(M)"]&\fbox{\Centerstack[l]{$W_3^{(\mu,3)}$}} \arrow[r,"D_{\mathfrak{g},\mathfrak{h}}(M)"] &  \fbox{\Centerstack[l]{$W_2^{(\mu,3)}$}} \arrow[r,"D_{\mathfrak{g},\mathfrak{h}}(M)"]& \fbox{\Centerstack[l]{$W_1^{(\mu,3)}$}}\arrow[r,"D_{\mathfrak{g},\mathfrak{h}}(M)"]&0\\
   	&\vdots&  \vdots& \vdots&
   \end{tikzcd}
   \vspace{-0.1mm}
   \hspace{-8.95cm}
   \begin{tikzcd}
   	\draw [line width=0.5pt,dash pattern=on 1pt off 1pt] (-9.85,2.5)-- (-9.85,-3);
   	\draw [line width=0.5pt,dash pattern=on 1pt off 1pt] (-9.85,2.5)-- (-8.3,2.5);
   	\draw [line width=0.5pt,dash pattern=on 1pt off 1pt] (-8.3,2.5)-- (-8.3,-3);
   	\draw [line width=0.5pt,dash pattern=on 1pt off 1pt] (-7.1,2.5)-- (-7.1,-3);
   	\draw [line width=0.5pt,dash pattern=on 1pt off 1pt] (-7.1,2.5)-- (-5.55,2.5);
   	\draw [line width=0.5pt,dash pattern=on 1pt off 1pt] (-5.55,2.5)-- (-5.55,-3);
   	\draw [line width=0.5pt,dash pattern=on 1pt off 1pt] (-4.35,2.5)-- (-4.35,-3);
   	\draw [line width=0.5pt,dash pattern=on 1pt off 1pt] (-4.35,2.5)-- (-2.8,2.5);
   	\draw [line width=0.5pt,dash pattern=on 1pt off 1pt] (-2.8,2.5)-- (-2.8,-3);
   \end{tikzcd}
\caption{Decomposition of $[\big(M\otimes S\big)(c)]_\mu$}
\end{figure}
   
	   Note that $N_{\mu,k}=N_{\mu,k}^+\oplus N_{\mu,k}^-$ while by convention $W_k^{(\mu,j)}=0$ whenever $k>m_{\mu,j}$. Then
	   \begin{equation*}
	   \big(\ker D_{\mathfrak{g},\mathfrak{h}}(M)_\mu^{2k+2}\big)^\pm=N_{\mu,1}^\pm\oplus\ldots \oplus N_{\mu,2k+2}^\pm
	   \end{equation*}
	   and
	   \begin{equation*}
	   \mathrm{im}\hspace{0.5mm} D_{\mathfrak{g},\mathfrak{h}}(M)_{\mid N_{\mu,2k+2}^\mp}=\bigoplus\limits_{\substack{j=1\text{ s.t.}\\ m_{\mu,j}\geq 2k+2}}^{n_\mu}(W_{2k+1}^{(\mu,j)})^\pm\cong N_{\mu,2k+2}^\mp
	   \end{equation*}
	   as vector spaces. Therefore
	 \begin{equation*}
	 {\renewcommand*{\arraystretch}{3}
	   \begin{array}{l}
	    \dfrac{N_{\mu,2k+1}^\pm}{\mathrm{im}\hspace{0.5mm}D_{\mathfrak{g},\mathfrak{h}}(M)_{\mid N_{\mu,2k+2}^\mp}}\cong \bigoplus\limits_{\substack{j=1\text{ s.t.}\\ m_{\mu,j}< 2k+2}}^{n_\mu}(W_{2k+1}^{(\mu,j)})^\pm=\bigoplus\limits_{\substack{j=1\text{ s.t.}\\ m_{\mu,j}= 2k+1}}^{n_\mu}(W_{2k+1}^{(\mu,j)})^\pm\\
	   \cong \dfrac{\big(\ker D_{\mathfrak{g},\mathfrak{h}}(M)^{2k+1}_\mu\big)^\pm}{\big(\ker D_{\mathfrak{g},\mathfrak{h}}(M)^{2k+1}_\mu\cap\mathrm{im}\hspace{0.5mm}D_{\mathfrak{g},\mathfrak{h}}(M)_\mu+\ker D_{\mathfrak{g},\mathfrak{h}}(M)^{2k}_\mu\big)^\pm}\cong H_\mathrm{top}^k(M)_\mu^\pm
	   \end{array}}
	   \end{equation*}
	   so that
	   \begin{align*}
	   \dim H^k_{\mathrm{top}}(M)_\mu^\pm&=\dim N_{\mu,2k+1}^\pm-\dim \mathrm{im}\hspace{0.5mm}D_{\mathfrak{g},\mathfrak{h}}(M)_{\mid N_{\mu,2k+2}^\mp}\\
	   &=\dim N_{\mu,2k+1}^\pm-\dim N_{\mu,2k+2}^\mp.
	   \end{align*}
	   Then 
	   \begin{align*}
	   	\dim H_{\mathrm{top}}(M)^+_\mu-\dim H_{\mathrm{top}}(M)^-_\mu=& \sum_{k\geq0}  \dim H^k_{\mathrm{top}}(M)_\mu^+-\sum_{k\geq0} \dim H^k_{\mathrm{top}}(M)_\mu^-\\
	   	=&\sum_{k\geq0}\big(\dim N_{\mu,2k+1}^+-\dim N_{\mu,2k+2}^-\big)\\
	   	&-\sum_{k\geq0}\big(\dim N_{\mu,2k+1}^--\dim N_{\mu,2k+2}^+\big)\\
	   	=&\sum_{k\geq0} \big(\dim N_{\mu,2k+1}^++\dim N_{\mu,2k+2}^+\big)\\
	   	&-\sum_{k\geq0} \big(\dim N_{\mu,2k+1}^-+\dim N_{\mu,2k+2}^-\big)\\
	   	=&\sum_{k\geq1}\dim N_{\mu,k}^+-\sum_{k\geq1}\dim N_{\mu,k}^-\\
	   	=&	\dim [\big(M\otimes S\big)(0)^+]_\mu-\dim[\big(M\otimes S\big)(0)^-]_\mu.
	   \end{align*}	   
	\end{proof}

The following theorem concerning the exact circle \eqref{exactcircle} is the analogue of \cite[Theorem 1.13]{somberg}. Nevertheless, the maps that we obtain are $\mathfrak{t}$-equivariant.

\begin{theorem}
	Let $M_1,M_2,M_3\in\mathcal{O}$ and 
	\begin{equation*}
	0\longrightarrow M_1\longrightarrow M_2\longrightarrow M_3\longrightarrow 0
	\end{equation*}
	be an exact sequence. Then there are maps such that 
	\vspace{3mm}
		\begin{equation}\label{againexact}
	\begin{tikzcd}[cramped, sep=huge]
	H_\mathrm{top}(M_1)^+\arrow[r] 
	& H_\mathrm{top}(M_2)^+ \arrow[r] &H_\mathrm{top}(M_3)^+ \arrow[d]\\
	H_\mathrm{top}(M_3)^-\arrow[u]
	&\arrow[l]H_\mathrm{top}(M_2)^- &\arrow[l]H_\mathrm{top}(M_1)^-
	\end{tikzcd}
	\end{equation}
	forms an exact circle of $\mathfrak{t}$-modules.
\end{theorem}

\begin{proof} We adapt the proof of \cite[Theorem 1.13]{somberg}; instead of $\mathfrak{h}$-isotypic components we consider finite dimensional weight subspaces.
	
Since tensorizing with a finite dimensional module is an exact functor, we obtain the exact sequence of $\mathfrak{h}$-modules:
\begin{equation}\label{exact}
0\longrightarrow M_1\otimes S\mathop{\longrightarrow}^{i}M_2\otimes S\mathop{\longrightarrow}^{p} M_3\otimes S\longrightarrow0.
\end{equation}
We note that the maps $i$ and $p$ of the above exact sequence commute with the action of $U(\mathfrak{g})\otimes \mathbf{C}(\mathfrak{q})$. As a consequence, they commute with the corresponding Dirac operators. Namely
\begin{align*}
i\circ D_{\mathfrak{g},\mathfrak{h}}(M_1)&=D_{\mathfrak{g},\mathfrak{h}}(M_2)\circ i\\
\text{and }\quad p\circ D_{\mathfrak{g},\mathfrak{h}}(M_2)&=D_{\mathfrak{g},\mathfrak{h}}(M_3)\circ p.
\end{align*}

Each $\mathfrak{h}$-module $M_i\otimes S$, $i=1,2,3$, admits a direct sum decomposition in terms of generalized eigenspaces for $D_{\mathfrak{g},\mathfrak{h}}(M_i)^2$ and in terms of weight subspaces while these decompositions are compatible.
Therefore, \eqref{exact} breaks up into a direct sum of exact sequences
\begin{equation*}
0\longrightarrow \big(M_1\otimes S\big)(c)_\mu\mathop{\longrightarrow}^{i}\big(M_2\otimes S\big)(c)_\mu\mathop{\longrightarrow}^{p} \big(M_3\otimes S\big)(c)_\mu\longrightarrow0,
\end{equation*}
where $\mu\in\mathfrak{t}^*$ and $\big(M_i\otimes S\big)(c)_\mu$ is the corresponding weight subspace of the generalized eigenspace with eigenvalue $c$. We are interested in the case of the generalized kernel of $D_{\mathfrak{g},\mathfrak{h}}(M_i)$, i.e. when $c=0$:

\begin{equation}\label{maria}
0\longrightarrow \big(M_1\otimes S\big)(0)_\mu\mathop{\longrightarrow}^{i}\big(M_2\otimes S\big)(0)_\mu\mathop{\longrightarrow}^{p} \big(M_3\otimes S\big)(0)_\mu\longrightarrow0.
\end{equation}
Since $\big(M_3\otimes S\big)(0)_\mu$ is finite dimensional, it can be decomposed into a direct sum of Jordan blocks ${}^{(3)}\hspace{-0.5mm}J_1^{\mu},\ldots, {}^{(3)}\hspace{-0.5mm}J_{n_\mu}^{\mu}$ for $D_{\mathfrak{g},\mathfrak{h}}(M_3)$. As we saw in the proof of Lemma \ref{lemdecomp}, every such block ${}^{(3)}\hspace{-0.5mm}J_j^{\mu}$ can be decomposed into a direct sum of $1$-dimensional subspaces
\begin{equation*}
{}^{(3)}\hspace{-0.5mm}J_j^{\mu}={}^{(3)}\hspace{-0.2mm}W_1^{(\mu,j)}\oplus\ldots\oplus{}^{(3)}\hspace{-0.2mm}W_{m_{\mu,j}}^{(\mu,j)}
\end{equation*}
such that 
  \begin{equation*}
D_{\mathfrak{g},\mathfrak{h}}(M_3){}^{(3)}\hspace{-0.2mm}W^{(\mu,j)}_1=0
\end{equation*}
and 
\begin{equation*}
D_{\mathfrak{g},\mathfrak{h}}(M_3):{}^{(3)}\hspace{-0.2mm}W_{i}^{(\mu,j)}\xrightarrow{\cong}{}^{(3)}\hspace{-0.2mm}W_{i-1}^{(\mu,j)} \hspace{0.5mm}\text{ for } i=2,\ldots,m_{\mu,j}.
\end{equation*}
For every top summand ${}^{(3)}\hspace{-0.2mm}W_{m_{\mu,j}}^{(\mu,j)}$, let ${}^{(2)}\hspace{-0.2mm}W_{l_{\mu,j}}^{(\mu,j)}$ be a preimage in $\big(M_2\otimes S)(0)_\mu$. Then, ${}^{(2)}\hspace{-0.2mm}W_{l_{\mu,j}}^{(\mu,j)}$ must be a top summand of a Jordan block of $\big(M_2\otimes S\big)(0)_\mu$. Indeed, if we suppose that there exists ${}^{(2)}\hspace{-0.2mm}W_{l_{\mu,j}+1}^{(\mu,j)}\neq 0$ such that 
\begin{equation*}
D_{\mathfrak{g},\mathfrak{h}}(M_2):{}^{(2)}\hspace{-0.2mm}W_{l_{\mu,j}+1}^{(\mu,j)}\xrightarrow{\cong}{}^{(2)}\hspace{-0.2mm}W_{l_{\mu,j}}^{(\mu,j)},
\end{equation*}
then 
\begin{equation*}
{}^{(3)}\hspace{-0.2mm}W_{m_{\mu,j}}^{(\mu,j)}=p\circ D_{\mathfrak{g},\mathfrak{h}}(M_2)\big({}^{(2)}\hspace{-0.2mm}W_{l_{\mu,j}+1}^{(\mu,j)}\big)=D_{\mathfrak{g},\mathfrak{h}}(M_3)\circ p\big({}^{(2)}\hspace{-0.2mm}W_{l_{\mu,j}+1}^{(\mu,j)}\big)=0
\end{equation*}
which is a contradiction. Hence, ${}^{(2)}\hspace{-0.2mm}W_{l_{\mu,j}}^{(\mu,j)}$ is the top summand of a Jordan block, say ${}^{(2)}\hspace{-0.5mm}J_j^\mu$, in $\big(M_2\otimes S\big)(0)_\mu$.
Then
\begin{equation*}
{}^{(1)}\hspace{-0.5mm}J_j^\mu:=\ker \big({}^{(2)}\hspace{-0.5mm}J_j^\mu\rightarrow {}^{(3)}\hspace{-0.5mm}J_j^\mu\big)
\end{equation*}
is a Jordan block of length $k_{\mu,j}:=l_{\mu,j}-m_{\mu,j}$ in $\big(M_1\otimes S\big)_\mu$, say of top summand ${}^{(1)}\hspace{-0.2mm}W_{k_{\mu,j}}^{(\mu,j)}$, and 
\begin{equation*}
0\longrightarrow {}^{(1)}\hspace{-0.5mm}J_j^\mu\mathop{\longrightarrow}^{i}{}^{(2)}\hspace{-0.5mm}J_j^\mu\mathop{\longrightarrow}^{p} {}^{(3)}\hspace{-0.5mm}J_j^\mu\longrightarrow0
\end{equation*}
forms a short exact sequence of vector spaces.

  \begin{figure}
	\begin{tikzcd}[row sep=small]	
		&{}^{(1)}\hspace{-0.5mm}J_j^\mu &  {}^{(2)}\hspace{-0.5mm}J_j^\mu & {}^{(3)}\hspace{-0.5mm}J_j^\mu &\\
		0\arrow[r]&0\arrow[r,"i"] &  \boxed{\substack{{}^{(2)}\hspace{-0.2mm}W_{l_{\mu,j}}^{(\mu,j)}\\ \downarrow \\ \vdots \\ \downarrow\\ {}^{(2)}\hspace{-0.2mm}W_{l_{\mu,j}-m_{\mu,j}+1}^{(\mu,j)}}} \arrow[r,"p"]& \boxed{\substack{{}^{(3)}\hspace{-0.2mm}W_{m_{\mu,j}}^{(\mu,j)}\\ \downarrow \\ \vdots\\ \downarrow\\ {}^{(3)}\hspace{-0.2mm}W_{1}^{(\mu,j)}}}\arrow[r]&0\\
		&&\downarrow&&\\
		0\arrow[r]&\boxed{\substack{{}^{(1)}\hspace{-0.2mm}W_{k_{\mu,j}}^{(\mu,j)}\\ \downarrow \\ \vdots\\ \downarrow\\ {}^{(1)}\hspace{-0.2mm}W_{1}^{(\mu,j)}}}\arrow[r] &  \boxed{\substack{{}^{(2)}\hspace{-0.2mm}W_{k_{\mu,j}}^{(\mu,j)}\\ \downarrow \\ \vdots \\ \downarrow\\ {}^{(2)}\hspace{-0.2mm}W_{1}^{(\mu,j)}}} \arrow[r,"p"]& 0\arrow[r]&0
	\end{tikzcd}
\caption{Decomposition of \eqref{maria} into compatible Jordan blocks}
\end{figure}

Then, either $k_{\mu,j}$, $l_{\mu,j}$ and $m_{\mu,j}$ are all even so that the corresponding Jordan blocks do not contribute to the higher Dirac cohomology, or exactly one of them is even while the other two are odd. Assume that $k_{\mu,j}$ is even. Then, 
\begin{align*}
H_\mathrm{top}({}^{(1)}\hspace{-0.5mm}J_j^\mu)&=0\\ H_\mathrm{top}({}^{(2)}\hspace{-0.5mm}J_j^\mu)&={}^{(2)}\hspace{-0.2mm}W_{l_{\mu,j}+1}^{(\mu,j)}\\
\text{and }\hspace{2mm} H_\mathrm{top}({}^{(3)}\hspace{-0.5mm}J_j^\mu)&={}^{(3)}\hspace{-0.2mm}W_{k_{\mu,j}+1}^{(\mu,j)}
\end{align*}
so that
\begin{equation*}
\begin{tikzcd}
& H_\mathrm{top}({}^{(1)}\hspace{-0.5mm}J_j^\mu) \arrow[dr,"0"] \\
H_\mathrm{top}({}^{(3)}\hspace{-0.5mm}J_j^\mu)\arrow[ur,"0"] && H_\mathrm{top}({}^{(2)}\hspace{-0.5mm}J_j^\mu)\arrow[ll,"1"]
\end{tikzcd}
\end{equation*}
forms an exact circle. In the same way, we treat the other two cases.
Therefore, applying the same procedure for the other Jordan blocks of $\big(M_3\otimes S\big)(0)_\mu$, we obtain an exact circle
\begin{equation*}
\begin{tikzcd}
& H_\mathrm{top}(M_1)_\mu \arrow[dr] \\
H_\mathrm{top}(M_3)_\mu\arrow[ur] && H_\mathrm{top}(M_2)_\mu\arrow[ll]
\end{tikzcd}
\end{equation*}
Taking direct sums over $\mu\in\mathfrak{t}^*$, we obtain an exact circle of $\mathfrak{t}$-modules
\begin{equation*}
\begin{tikzcd}
& H_\mathrm{top}(M_1) \arrow[dr] \\
H_\mathrm{top}(M_3)\arrow[ur] && H_\mathrm{top}(M_2)\arrow[ll]
\end{tikzcd}
\end{equation*}
while considering the natural $\mathbb{Z}_2$-gradings, we obtain the exact circle \eqref{againexact}.
\end{proof}

Finally, we can state the analogue of \cite[Theorem 1.10]{somberg}.

\begin{theorem}
	Let $M\in\mathcal{O}_{\chi_\lambda}$. Then $H_{\mathrm{top}}(M)\in\mathcal{O}(\mathfrak{h})_{\xi_\lambda}$.
\end{theorem}

\begin{proof}
	The proof is similar to the proof of \cite[Theorem 1.10]{somberg} and uses the argument developed in \cite{K3,pandzic,DH} that we expained in the proof of Corollary \ref{vogy}. Namely, with the notation used in the proof of Corollary \ref{vogy}, it suffices to show that $D_{\mathfrak{g},\mathfrak{h}}(M)a +aD_{\mathfrak{g},\mathfrak{h}}(M)$ of expression \eqref{katokato} acts by zero on $H_{\mathrm{top}}(M)$.
	Let $x\in \ker D_{\mathfrak{g},\mathfrak{h}}(M)^{2k+1}$ for some $k\geq0$. Then 
	\begin{equation*}
	D_{\mathfrak{g},\mathfrak{h}}(M)^{2k}aD_{\mathfrak{g},\mathfrak{h}}(M)x=aD_{\mathfrak{g},\mathfrak{h}}(M)^{2k+1}x=0
	\end{equation*}
	so that $aD_{\mathfrak{g},\mathfrak{h}}(M)x\in\ker D_{\mathfrak{g},\mathfrak{h}}(M)^{2k}$ while
	\begin{equation*}D_{\mathfrak{g},\mathfrak{h}}(M)^{2k+1}D_{\mathfrak{g},\mathfrak{h}}(M)ax=D_{\mathfrak{g},\mathfrak{h}}(M)^{2(k+1)}ax=aD_{\mathfrak{g},\mathfrak{h}}(M)^{2(k+2)}x=0\end{equation*}
	so that 
	$D_{\mathfrak{g},\mathfrak{h}}(M)ax\in \ker D_{\mathfrak{g},\mathfrak{h}}(M)^{2k+1}\cap\mathrm{im}\hspace{0.5mm}D_{\mathfrak{g},\mathfrak{h}}(M)$.
	Hence $D_{\mathfrak{g},\mathfrak{h}}(M)a+aD_{\mathfrak{g},\mathfrak{h}}(M)$ acts by zero on $H_{\mathrm{top}}(M)$.
\end{proof}

%\begin{corollary}
%	Let $M(\lambda)$ be the Verma module of highest weight $\lambda$. Then, in the Grothendieck's group
%	\begin{equation*}
%		H_{\mathrm{top}}(M(\lambda))=\sum[M(\lambda):L(w\cdot\lambda)] H_D(L(w\cdot\lambda))
%	\end{equation*}
%\end{corollary}

\end{document}